\documentclass{amsart}

\usepackage{color}

\newtheorem{theorem}{Theorem}[section]
\newtheorem{lemma}{Lemma}
\newtheorem{corollary}{Corollary}
\newtheorem{proposition}{Proposition}
\theoremstyle{conjecture}
\newtheorem{conjecture}{Conjecture}
\theoremstyle{definition}

\newtheorem{question}{Question}
\theoremstyle{question}
 \usepackage[latin1]{inputenc}

\theoremstyle{remark}
\newtheorem{remark}{Remark}
\theoremstyle{remarks}
\newtheorem{remarks}{Remarks}
\theoremstyle{example}
\newtheorem{example}{Example}
\numberwithin{equation}{section}

\theoremstyle{proofidea}
\newtheorem*{proofidea}{Idea of the Proof}

\begin{document}
\title{Loops in $SU(2)$ and Factorization, II}

\author{Estelle Basor}
\email{ebasor@aimath.org}

\author{Doug Pickrell}
\email{pickrell@math.arizona.edu}

\begin{abstract} In the prequel to this paper, we proved that for a $SU(2,\mathbb C)$ valued loop
having the critical degree of smoothness (one half of a derivative in the $L^2$ Sobolev sense),
the following statements are equivalent: (1) the Toeplitz and shifted Toeplitz operators associated to the loop
are invertible, (2) the loop has a unique triangular factorization, and (3) the loop has a unique root subgroup
factorization. The analytic aspects of these equivalences hinge on factorization formulas for determinants of Toeplitz operators. In addition to discussing some consequences, the main point of this sequel is to discuss generalizations to measurable loops, in particular loops of vanishing mean oscillation. The VMO generalization hinges on an operator-theoretic factorization for Toeplitz operators, in lieu of factorization for determinants.

\end{abstract}
\maketitle

\setcounter{section}{-1}

\section{Introduction}\label{Introduction}

This paper concerns the Polish topological groups of maps $W^{1/2}(S^1,SU(2))$, $VMO(S^1,SU(2))$, and $Meas(S^1,SU(2))$ (equivalence classes of $SU(2,\mathbb C)$ valued loops which have one half of a derivative in the $L^2$ Sobolev sense, are of vanishing mean oscillation, and are Lebesgue measurable, respectively; the basic background - such as the Polish topologies of these groups - is recalled in Section \ref{notation}). In an attempt to motivate the subject matter, we first consider a broader perspective.

Suppose that $K$ is a compact Lie group. The equatorial inclusions
$$ S^0 \subset S^1 \subset S^2 \subset S^3 \subset ... $$
induce (down arrow) inclusions and (left to right arrow) trace homomorphisms of groups
$$ \begin{matrix}...& \to & C^{\infty}(S^3,K)& \to & C^{\infty}(S^2,K) & \to & C^{\infty}(S^1,K)&\to &C^{\infty}(S^0,K)\\
\downarrow & & \downarrow & &\downarrow & &\downarrow & &&\\
...& \to &W^{3/2}(S^3,K)& \to &W^1(S^2,K) &\to & W^{1/2}(S^1,K)&&\\
\downarrow & & \downarrow & &\downarrow & &\downarrow &&\\
...& \to & VMO(S^3,K)& \to & VMO(S^2,K) & \to & VMO(S^1,K)&&&\\
\downarrow & & \downarrow & &\downarrow & &\downarrow &&&\\
...&  & Meas(S^3,K)&  & Meas(S^2,K) &  & Meas(S^1,K)&&&
\end{matrix}$$
The groups of smooth maps are Frechet Lie groups (see Section 3.2 of \cite{PS}), hence it is known what they look like locally, and their global topology can be analyzed using conventional methods of algebraic topology.

For the groups $W^{d/2}(S^d,K)\subset VMO(S^d,K)\subset Meas(S^d,K)$, generic group elements are not continuous mappings (Recall that $s=d/2$ is the critical $L^2$ exponent: the Sobolev embedding $W^{s,L^2}(S^d)\to C^0(S^d)$ holds for $s>d/2$ and marginally fails for $s=d/2$). The usual approach to understanding the local structure of continuous mapping groups is to fix a proper open coordinate neighborhood of $1\in K$ (homeomorphic to $\mathbb R^n$, say) and consider the set of maps with image in this neighborhood. This fails in our context because generic group elements in this set are locally unbounded, and hence this set is not an open neighborhood of $1\in W^{d/2}(S^d,K)$ (or $VMO$, or $Meas$). For similar reasons conventional methods of algebraic topology do not apply to understand the global topology. This is problematic, because it is important to understand the local and global
topology of these (Polish) mapping groups; see \cite{BN1}, \cite{BN2}, \cite{Brezis}, and references, for foundational work in this direction and further motivation. The simplest hypothesis - this is pure speculation - is that
for all $d\ge 1$, $W^{d/2}(S^d,K)$ and $VMO(S^d,K)$ are topological manifolds (they are definitely not smooth Lie groups as Polish topological groups), and the inclusions
\begin{equation}\label{homotopyequiv}C^{\infty}(S^d,K) \to W^{d/2}(S^d,K) \to VMO(S^d,K)\end{equation}
are homotopy equivalences.  This is exemplified by the existence of trace maps for VMO (see \cite{BN2}, and note we are considering an equatorial trace) and the nonexistence of trace maps for measurable maps in the above diagram. More directly relevant to this paper,
in the elemental case $d=1$, the global topology for the smooth loop space is intimately related to the map
\begin{equation}\label{Fredholm}C^{\infty}(S^1,K)\to Fred(H_+):g \to A(g)\end{equation} where $A(g)$ is the Toeplitz operator with symbol $g$ (see chapter 6 of \cite{PS}); the point is that $VMO(S^1,K)$ is the natural domain (see Proposition \ref{opertopology} below for a more precise statement).

\begin{remark} $Meas(S^d,K)$ is an outlier in this topological digression. Since its definition depends only upon the Lebesgue measure class of $S^d$, it is isomorphic to $Meas([0,1],K)$, and it is contractible.
\end{remark}

In this paper $d=1$, unless noted otherwise. In this case the claim about the homotopy equivalences basically follows from the Grassmannian model approach in chapter 8 of \cite{PS} (with modifications which we will note). We are mainly interested in technology which is useful in understanding the local structure. We will focus on $K=SU(2)$, rather than a general compact Lie group, because the main issues are more analytic than Lie theoretic (see \cite{PP} for the general Lie theoretic framework). In the prequel to this paper, we showed that for $g\in W^{1/2}(S^1,SU(2))$, the following statements are equivalent: (1) the Toeplitz and shifted Toeplitz operators associated to $g$ are invertible, (2) $g$ has a unique triangular factorization, and (3) $g$ has a unique root subgroup factorization (we will review this in Section \ref{W1/2}). This is a statement about the (open) top stratum of the $W^{1/2}$ loop group, and there is a  generalization to the finite codimensional lower strata. The key to the equivalence of (1)-(3), and in truth the more interesting point, is that there exists an explicit factorization for $det(A(g)A(g^{-1}))$, akin to the Plancherel formula in linear Fourier analysis (see (\ref{Toeplitz0})). A corollary of this is that $W^{1/2}(S^1,SU(2))$ is a nonsmooth topological manifold modeled on $l^2$, and it is homotopy equivalent to the smooth loop group.

\begin{remark} The scalar $det(A(g)A(g^{-1}))$ appears prominently in Harold Widom's
landmark paper \cite{Widom}, as the constant term in the expansion of determinants of block Toeplitz matrices for
symbols that are bounded and in $W^{1/2}$. This paper not only gave the asymptotics in
the block case, but paved the way for operator theory and Banach algebra approaches
for the asymptotic expansions for determinants of structured operators. This constant
is related to quantities that appear in the theory of tau-functions, dimer-models,
random matrix theory, and other areas of mathematical physics and is now commonly
called Widom's constant.
\end{remark} 

The main point of this paper is to investigate extensions of this theory to VMO (and more general Besov spaces which interpolate between $W^{1/2}$ and VMO, following Peller), and some qualified extensions to the measurable (or $L^2$) context. In the VMO context, the Toeplitz operator $A(g)$ is Fredholm, the determinant $det(A(g))$ makes sense as a section of a determinant line bundle, but the scalar expression $det(A(g)A(g^{-1}))$ is identically zero
in the complement of $W^{1/2}(S^1,SU(2))$. Roughly speaking the theory extends because, as we essentially observed in \cite{BP} (we will need a slight refinement), there is actually a factorization of $A(g)$, as an operator, in root subgroup coordinates.

\begin{remark}\label{factremark} In the notation of Theorem 1.4 of \cite{BP}, we will show that
$$A(k_1^*\left( \begin{matrix}e^{\chi}&0\\0&e^{-\chi}\end{matrix}\right) k_2)= A(k_1^* \left( \begin{matrix}e^{\chi_-}&0\\0&e^{-\chi_-}\end{matrix}\right)) A(\left( \begin{matrix}e^{\chi_0}&0\\0&e^{-\chi_0}\end{matrix}\right))A(\left( \begin{matrix}e^{\chi_+}&0\\0&e^{-\chi_+}\end{matrix}\right)k_2) $$
and we will see that it is relatively easy to analyze the Toeplitz operators on the right hand side. If we had observed this in \cite{BP},
then we could have eliminated Lemmas 4 and 5 in Section 5.1, and this would have greatly simplified the exposition.
The caveat in the context of this paper is that for $\chi\in VMO(S^1,i\mathbb R)$, while $exp(\chi)$ is bounded and
$\chi_{\pm}\in VMO(S^1)$, $exp(\chi_{\pm})$ are not generally bounded. Thus there are domain issues which we have suppressed.

\end{remark}

\subsection{Plan of the Paper}

In Section \ref{notation} we establish basic notation and recall some background results, especially the operator theoretic realization of the topologies for the various spaces of loops.

In the first part of Section \ref{W1/2} we succinctly outline the main results from \cite{Pi1} for loops into $SU(2)$ with critical degree of smoothness in the $L^2$ Sobolev sense (the $W^{1/2}$ theory). These results hinge on various Plancherel-esque identities involving block Toeplitz determinants, extending identities of Szego and Widom. In the second part of the section we discuss coordinates for $ W^{1/2}(S^1,SU(2))$ as a topological manifold, using root subgroup factorization.  In the third part of the section we outline how the Birkhoff stratification of the smooth loop space extends to $ W^{1/2}(S^1,SU(2))$.

In Section \ref{L2} we consider measurable maps, which we refer to as the $L^2$ theory. Here we are probing
the edge of deterministic results.  For a measurable map into $SU(2)$, the Toeplitz operator is not in general Fredholm. Uniqueness in root subgroup factorization is lost because of the existence of singular inner functions. Some implications need to be formulated in almost sure terms. Beyond the narrow aims of this paper, our main goal is actually to understand probabilistic statements which are just beyond this edge (\cite{Pi3}).

In Section \ref{VMO} we consider maps of vanishing mean oscillation, and more generally maps satisfying a Besov
condition $B_p^{1/p}$ (which interpolates between $W^{1/2}$ and $VMO$). As in the case of $W^{1/2}$ loops, by specializing our $L^2$ results, we show that $VMO(S^1,SU(2))$ is a topological manifold and (following the ideas of \cite{PS}) we note that the inclusion of smooth loops is a homotopy equivalence.

There are two appendices in which we mention some combinatorial and analytic issues which we have not been able to resolve.

\section{Notation and Background}\label{notation}

If $f(z)=\sum f_nz^n$, then we will write
$$f=f_-+f_0+f_+$$
where $f_-(z)=\sum_{n<0}f_nz^n$ and
$f_+(z)=\sum_{n>0}f_nz^n$, $f_{-0}=f_-+f_0$, $f_{0+}=f_0+f_+$, and $f^*(z)=\sum (f_{-n})^*z^n$,
where $w^*=\bar{w}$ is the complex
conjugate of the complex number $w$.
If the Fourier series is convergent at a point $z\in S^1$, then $f^*(z)$ is the conjugate of the complex number $f(z)$.  If $f \in
H^0(\Delta)$, then $f^* \in H^0(\Delta^*)$, where $\Delta$ is the
open unit disk, $\Delta^*$ is the open unit disk at $\infty$, and
$H^0(U)$ denotes the space of holomorphic functions for a domain $U\subset \mathbb C$.

$W^{1/2}(S^1,\mathbb C)$ denotes the Hilbert space of (equivalence classes of Lebesgue) measurable functions $f(z)$ which have half a derivative in the $L^2$ Sobolev sense; the precise form of the norm is not important, but one possibility is
$$|f|_{W^{1/2}}=(\sum_{n=-\infty}^{\infty}(1+n^2)^{1/2}|\widehat f(n)|^2)^{1/2}$$
where $\widehat f$ denotes the Fourier transform.
$VMO(S^1)$ denotes the Banach space of (equivalence classes of Lebesgue) measurable functions which are of vanishing mean oscillation, or equivalently the closure of the subspace of continuous functions in BMO; a typical choice of norm is
$$|f|_{BMO} =|\widehat f(0)|+\sup_{I} \frac{1}{|I|}\int_I|f-f_I|d\theta $$
where $f_I$ denotes the average of $f$ over the arc $I$, and the supremum is over all arcs in $S^1$. $Meas(S^1,\mathbb C)$ denotes equivalence classes of Lebesgue measurable functions with the topology corresponding to convergence in (Lebesgue) measure; this is induced by a complete separable metric, see below. Besov spaces $B^{1/p}_p$ which interpolate between $W^{1/2}$ and $VMO$ for $2\le p \le \infty$ will be used below and in Section \ref{VMO} (for this we will refer to chapter 6 and Appendix 2 of \cite{Peller}).

$\mathbf w^{1/2}$ denotes
the Hilbert space of complex sequences $\zeta$ such that $\sum_{k=1}^{\infty}k|\zeta_k|^2<\infty$.

$L_{fin}SU(2)$ ($L_{fin} SL(2,{\mathbb C})$) denotes the group
consisting of functions $S^1 \to SU(2)$ ($SL(2,\mathbb C)$,
respectively) having finite Fourier series, with pointwise
multiplication. For example, for $\zeta \in \mathbb C$ and $n\in
\mathbb Z$, the function
$$S^1 \to SU(2):z \to
\mathbf a(\zeta) \left(\begin{matrix} 1&
\zeta z^{-n}\\
-\bar{\zeta}z^n&1\end{matrix} \right),$$ where $\mathbf
a(\zeta)=(1+\vert \zeta \vert ^2)^{-1/2}$, is in $L_{fin}SU(2)$.

As in the introduction, consider the groups
$$W^{d/2}(S^d,SU(2))\subset VMO(S^d,SU(2))\subset Meas(S^d,SU(2))$$

\begin{remark} (A digression) In the topologies induced
by the Banach algebras $L^{\infty}\cap W^{d/2}$, $QC:=L^{\infty}\cap VMO$, and $L^{\infty}$, respectively, these are Banach Lie groups. However in all cases smooth loops are not dense, and there are uncountably many connected components. Convergence in each of these Banach algebras implies uniform convergence, hence the identity component in each case consists of classes which have continuous representatives. This is interesting: for $X$ of dimension $d\ge 1$ and not necessarily compact, the Banach algebra $BC\cap W^{d/2}(X)$ ($BC$ stands for bounded continuous) is closely related to the notion of a Royden algebra; it is conjecturally a quasiconformal invariant of $X$ (see \cite{Lewis}).
However this is {\bf not} what we are interested in.

\end{remark}

In this paper we will always view $W^{d/2}(S^d,SU(2))$, $VMO(S^d,SU(2))$, and $Meas(S^d,SU(2))$ as topological groups with the complete separable (Polish) topologies induced by $W^{d/2}$, $VMO$, and convergence in probability, respectively. For measurable maps there is a well-known way to represent the topology using operator methods: the bijection
$$Meas(S^d,U(2)) \to \{\text{unitary multiplication operators on }L^2(S^d,\mathbb C^2)\}$$
is a homeomorphism with respect to the convergence in probability topology and the strong (or weak)
topology for unitary multiplication operators (see Section 2 of \cite{Moore}). For the other mapping groups, essentially following \cite{PS}, we will substitute restricted unitary groups (see below).

Now suppose that $d=1$. In this setup the inclusions
$$L_{fin}SU(2)\subset C^{\infty}(S^1,SU(2))\subset W^{1/2}(S^1,SU(2)) \subset VMO(S^1,SU(2)) \subset Meas(S^1,SU(2))$$
are dense. The first three inclusions are homotopy equivalences (for the first inclusion this follows from Proposition 5.2.5 of \cite{PS}; for the second and third inclusions, this follows from the Grassmannian model approach in chapter 8 of \cite{PS}, as we will explain below). One of the goals of this paper is to understand
some analytic aspects of this in a concrete way. The fourth inclusion is a map into a contractible space.

Suppose that $g\in L^1(S^1,SL(2,{\mathbb C}))$. A triangular
factorization of $g$ is a factorization of the form
\begin{equation}\label{factorization}g=l(g)m(g)a(g)u(g),\end{equation}
where
\[l=\left(\begin{array}{cc}
l_{11}&l_{12}\\
l_{21}&l_{22}\end{array} \right)\in H^0(\Delta^{*},SL(2,{\mathbb
C})),\quad l(\infty )=\left(\begin{array}{cc}
1&0\\
l_{21}(\infty )&1\end{array} \right),\] $l$ has a $L^2$ radial
limit, $m=\left(\begin{array}{cc}
m_0&0\\
0&m_0^{-1}\end{array} \right)$, $m_0\in S^1$,
$a(g)=\left(\begin{array}{cc}
a_0&0\\
0&a_0^{-1}\end{array} \right)$, $a_0>0$,
\[u=\left(\begin{array}{cc}
u_{11}&u_{12}\\
u_{21}&u_{22}\end{array} \right)\in H^0(\Delta ,SL(2,{\mathbb
C})),\quad u(0)=\left(\begin{array}{cc}
1&u_{12}(0)\\
0&1\end{array} \right),\] and $u$ has a $L^2$ radial limit. Note
that (\ref{factorization}) is an equality of measurable functions
on $S^1$. A Birkhoff (or Wiener-Hopf, or Riemann-Hilbert)
factorization is a factorization of the form $g=g_-g_0g_+$, where
$g_-\in H^0(\Delta^*,\infty;SL(2,\mathbb C),1)$, $g_0\in
SL(2,\mathbb C)$, $g_+\in H^0(\Delta,0;SL(2,\mathbb C),1)$, and
$g_{\pm}$ have $L^2$ radial limits on $S^1$. Clearly $g$ has a
triangular factorization if and only if $g$ has a Birkhoff
factorization and $g_0$ has a triangular factorization, in the
usual sense of matrices.

As in \cite{PS}, consider the polarized Hilbert space
\begin{equation}\label{polarization} H:=L^2(S^1,C^2)=H_{+}\oplus H_{-},\end{equation} where
$H_{+}=P_{+}H$
consists of $L^2$-boundary values of functions holomorphic in
$\Delta$. If $g\in L^{\infty}(S^1,SL(2,\mathbb C))$, we write the
bounded multiplication operator defined by $g$ on $H$ as
\begin{equation}\label{multiplicationop}M_g=\left(\begin{array}{cc}
A(g)&B(g)\\
C(g)&D(g)\end{array} \right)\end{equation} where
$A(g)=P_{+}M_gP_{+}$ is the (block) Toeplitz operator associated
to $g$ and so on. If $g$ has the Fourier expansion $g=\sum
g_nz^n$,
$g_n=\left(\begin{matrix}a_n&b_n\\c_n&d_n\end{matrix}\right)$,
then relative to the basis for $H$
\begin{equation}\label{basis} ..
\epsilon_1z,\epsilon_2z,\epsilon_1,\epsilon_2,\epsilon_1z^{-1},\epsilon_2z^{-1},..\end{equation}
where $\{\epsilon_1,\epsilon_2\}$ is the standard basis for
$\mathbb C^2$, the matrix of $M_g$ is block periodic of the form
\begin{equation}\label{matrix}\begin{array}{ccccccccc}&.&.&.&.&.&.&.&\\..&a_0&b_0&a_1&b_1&\vert&a_2&b_2&..
\\..&c_0&d_0&c_1&d_1&\vert&c_{2}&d_{2}&..\\
..&a_{-1}&b_{-1}&a_0&b_0&\vert&a_1&b_1&.. \\
..&c_{-1}&d_{-1}&c_0&d_0&\vert&c_1&d_1&..\\-&-&-&-&-&-&-&-&-\\
..&a_{-2}&b_{-2}&a_{-1}&b_{-1}&\vert&a_0&b_0&..\\
..&c_{-2}&d_{-2}&c_{-1}&d_{-1}&\vert&c_0&d_0&..\\&.&.&.&.&.&.&.&
\end{array}\end{equation}
From this matrix form, it is clear that, up to equivalence, $M_g$
has just two types of ``principal minors", the matrix representing
$A(g)$, and the matrix representing the shifted Toeplitz operator
$A_1(g)$, the compression of $M_g$ to the closed subspace spanned by
$\{\epsilon_iz^j:i=1,2,j>0\}\cup\{\epsilon_1\}$.

Given the polarization $H=H_+\oplus H_-$ and a
symmetrically normed ideal $\mathcal I\subset \mathcal L(H)$, there is an associated Banach
$*$-algebra, $\mathcal L_{(\mathcal I)}$, which consists of bounded operators on $
H$,
represented as two by two matrices as in (\ref{multiplicationop}) such that $B,C\in \mathcal I$ with
the norm
\begin{equation}\label{norm}\vert\left(\begin{matrix} A&\\&D\end{matrix} \right)\vert_{\mathcal L}+\vert\left
(\begin{matrix} &B\\C&\end{matrix} \right)\vert_{\mathcal I}\end{equation}
and the usual $*$-operation.  The
corresponding unitary group is
$$U_{(\mathcal I)}=U(H)\cap \mathcal L_{(\mathcal I)};$$
it is referred to as a restricted unitary group in \cite{PS}.  There
are two standard topologies on $U_{(\mathcal I)}$.  The first is the induced
Banach topology, and in this topology $U_{(\mathcal I)}$ has the additional
structure of a Banach Lie group. The second topology, the one we will always use, is the Polish
topology for which convergence means that for
$g_n,g\in U_{(\mathcal I)}$, $g_n\to g$ if and only if $g_n\to g$ strongly and
$$\left(\begin{matrix} &B_n\\C_n&\end{matrix} \right)\to\left(\begin{matrix} &B\\
C&\end{matrix} \right)\quad in\quad \mathcal I$$

\begin{remark}For the unitary group of a countably infinite dimensional Hilbert space,
the group of unitary operators with either the strong or with the operator norm topology is contractible.
Consequently the algebraic topology of $U_{(\mathcal I)}$ is the same for the first and second topologies.
But we are always interested in the second (Polish) topology.
\end{remark}

In the following proposition $\mathcal L_p$ refers to the Schatten ideal, and the Besov space $B_p^{1/p}$
is reviewed in chapter 6 of \cite{Peller}.

\begin{proposition}\label{opertopology} (a) For the Hardy polarization (\ref{polarization}) and $g\in L^{\infty}(S^1,\mathcal L(\mathbb C^2))$, $g\in \mathcal L_{(\mathcal L_p)}$ iff $g$ belongs to the Besov space $B_p^{1/p}$ for $p<\infty$ and $VMO$ for $p=\infty$.

(b) For $p<\infty$, $B_p^{1/p}(S^1,K) \to U_{(\mathcal L_p)}(H_+\oplus H_-) $ is a homeomorphism onto its image;
in particular

(b') $W^{1/2}(S^1,K) \to U_{(\mathcal L_2)}(H_+\oplus H_-) $ is a homeomorphism onto its image.

(c) $VMO(S^1,K) \to U_{(\mathcal L_{\infty})}(H_+\oplus H_-) $ is a homeomorphism onto its image.

(d) $U_{(\mathcal L_{\infty})}(H_+\oplus H_-) \to Fred(H_+)$ is a homotopy equivalence.

\end{proposition}

\begin{proof} Part (a) is due to Peller for $0<p<\infty$ (see chapter 6 of \cite{Peller}) and Hartman for $p=\infty$ (see pages 27-28 of chapter 1 of \cite{Peller}). Note that when only considering the Hankel operators $B$ and $C$, one can relax the boundedness hypothesis on $g$ to $g\in L^2$.

Part (b') follows by inspecting the Hilbert-Schmidt properties of the matrix (\ref{matrix}). Parts (b) and (c)
are basically implicit in the results of Peller and Hartmann, see chapters 6 and 1 of \cite{Peller}, respectively. The images in parts (b) and (c), for $K=U(2)$,
are described in Proposition (6.3.3) of \cite{PS}. Part (d) is essentially Proposition (6.2.4) of \cite{PS}.
\end{proof}

Given a countably infinite dimensional Hilbert space such as $H_+$, Quillen constructed a holomorphic determinant line bundle $Det\to Fred(H_+)$ and a canonical holomorphic section $det$ which vanishes on the complement of invertible operators. This induces a determinant bundle
$$A^*Det \to VMO(S^1,SU(2)) $$
(There is a discussion of this, and references, at the end of Section 7.7 of \cite{PS}). This is an elegant way to think about the following corollary, but there is also a simple proof using the operator-theoretic realization of the VMO topology.

\begin{corollary}\label{VMOinvert_ops} For $VMO(S^1,SU(2)) $ the set of loops with invertible Toeplitz operators is defined by the equation
$det(A(g))\ne 0$, hence is open. The same applies for the shifted Toeplitz operator.
\end{corollary}

\begin{proof} Suppose that $g_n\in VMO(S^1,SU(2))$ converges in VMO to $g$ and $A(g)$ is invertible. We must show that $A(g_n)$ is invertible for large $n$.
$$A(g_n)A(g_n^{-1})=1-B(g_n)C(g_n^{-1})=1-B(g_n)B(g_n)^*$$
By part (c) of the preceding proposition, this converges uniformly to $A(g)A(g^{-1})=A(g)A(g)^*=1-B(g)B(g)^*$, which is invertible. This implies that $A(g_n)A(g_n^{-1})$ is invertible for large $n$, hence $A(g_n)$ is invertible for large $n$. \end{proof}

\begin{remark} For $Meas(S^1,SU(2)) $, or even for its diagonal subgroup $\{\left(\begin{matrix}\lambda(z)&0\\0&\lambda(z)^{-1}\end{matrix}\right)$, the set of loops with invertible Toeplitz operators is NOT open. To see this let $\lambda_n=exp(f_n):S^1 \to S^1$ be a continuous loop which rapidly winds once
around the circle in the interval $[0,1/n]$, and equals $1$ otherwise (this is called a blip). This has degree one, hence the Toeplitz operator $\dot A(\lambda)$ has Fredholm index $-1$ and is not
invertible for all $n$. Nonetheless $\lambda_n \to 1$ in measure.

This line of argument does not apply to $VMO(S^1,S^1)$, because degree is well-defined, continuous and separates the group into path connected components - this is the main point of \cite{BN1}.
\end{remark}

\section{The $W^{1/2}$ Theory}\label{W1/2}

The first part of this section is a succinct review of relevant results from \cite{Pi1}. The subsequent
subsections describe some consequences.

\begin{theorem}\label{introtheorem1}  Suppose that $k_1:S^1\to SU(2)$ is Lebesgue measurable. The
following are equivalent:

(I.1) $k_1 \in W^{1/2}(S^1,SU(2))$ and is of the form
$$k_1(z)=\left(\begin{matrix} a(z)&b(z)\\
-b^*(z)&a^*(z)\end{matrix} \right),\quad z\in S^1,$$ where $a,b\in H^0(\Delta)$, $a(0)>0$,
and $a$ and $b$ do not
simultaneously vanish at a point in $\Delta$.

(I.2) $k_1$ has a (root subgroup) factorization, in the sense that
$$k_1(z)=\lim_{n\to\infty}\mathbf a(\eta_n)\left(\begin{matrix} 1&-\bar{\eta}_nz^n\\
\eta_nz^{-n}&1\end{matrix} \right)..\mathbf
a(\eta_0)\left(\begin{matrix} 1&
-\bar{\eta}_0\\
\eta_0&1\end{matrix} \right)$$ for a.e. $z\in S^1$, where $(\eta_i)\in \mathbf w^{1/2}$
and the limit is understood in the $W^{1/2}$ sense.

(I.3) $k_1$ has triangular factorization of the form
$$\left(\begin{matrix} 1&0\\
y^*(z)&1\end{matrix} \right)\left(\begin{matrix}\mathbf a_1&0\\
0&\mathbf a_1^{-1}\end{matrix} \right)\left(\begin{matrix} \alpha_1 (z)&\beta_1 (z)\\
\gamma_1 (z)&\delta_1 (z)\end{matrix} \right),$$ where $\mathbf a_1>0$, $y=\sum_{j=0}^{\infty} y_jz^j$ and $\alpha_1(z),\beta_1(z)\in W^{1/2}$.

Suppose that $k_2:S^1 \to SU(2)$ is Lebesgue measurable. The following are equivalent:

(II.1) $k_2\in W^{1/2}(S^1,SU(2))$ and is of the form
$$k_2(z)=\left(\begin{matrix} d^{*}(z)&-c^{*}(z)\\
c(z)&d(z)\end{matrix} \right),\quad z\in S^1,$$ where $c,d\in
H^0(\Delta)$, $c(0)=0$, $d(0)>0$, and $c$ and $d$ do not
simultaneously vanish at a point in $\Delta$.

(II.2) $k_2$ has a (root subgroup) factorization of the form
$$k_2(z)=\lim_{n\to\infty}\mathbf a(\zeta_n)\left(\begin{matrix} 1&\zeta_nz^{-n}\\
-\bar{\zeta}_nz^n&1\end{matrix} \right)..\mathbf
a(\zeta_1)\left(\begin{matrix} 1&
\zeta_1z^{-1}\\
-\bar{\zeta}_1z&1\end{matrix} \right)$$ for a.e. $z\in S^1$, where $(\eta_i)\in \mathbf w^{1/2}$
and the limit is understood in the $W^{1/2}$ sense.

(II.3) $k_2$ has triangular factorization of the form
$$\left(\begin{matrix} 1&x^*(z)\\
0&1\end{matrix} \right)\left(\begin{matrix}\mathbf a_2&0\\
0&\mathbf a_2^{-1}\end{matrix} \right)\left(\begin{matrix} \alpha_2 (z)&\beta_2 (z)\\
\gamma_2 (z)&\delta_2 (z)\end{matrix} \right)$$
where $\mathbf a_2>0$, $x=\sum_{j=1}^{\infty}x_jz^j$, and $\gamma_2(z),\delta_2(z)\in W^{1/2}$.

\end{theorem}

\begin{remark}\label{PSUequivariance} There is a $PSU(1,1)$-equivariant Frechet space isomorphism
\begin{equation}H^0(\Delta)/\mathbb C \stackrel{\partial}{\rightarrow} H^1(\Delta):f_+\to\Theta :=\partial
f_+.\label{1.9}\end{equation} This representation is essentially unitary, where the norm of $f_+$ is the square root of $\int \partial f_+\wedge *\overline{\partial f_+}$. To say that $f_+(z)\in W^{1/2}(S^1)$ and is holomorphic in $\Delta$ is equivalent to saying that $\partial f_+\in H^1(\Delta)$ and square integrable (in the natural sense which we have just defined). This comment applies to the conditions we are imposing on $a_1,b_1,c_2,d_2,x,y$ in the statement of the theorem.

\end{remark}

\begin{proofidea} For $k_2\in L_{fin}SU(2) $, these correspondences are algebraic.
To be more precise, given a sequence $\zeta$ as in II.2 with a finite number of nonzero terms,
there are explicit polynomial expressions
for $x$, $\alpha_2$, $\beta_2$, $\gamma_2$ and $\delta_2$, and
\begin{equation}\label{l2identity}\mathbf a_2^2=\prod_{k>0} (1+|\zeta_k|^2) \end{equation}
Conversely, given $k_2$ as in II.1 or II.3, the sequence $\zeta$ can be recovered recursively from the Taylor expansion
\begin{equation}\label{Taylorexp}(c_2/d_2)(z)=(\gamma_2/\delta_2)(z)=(-\overline{\zeta}_1) z+(-\overline{\zeta}_2)(1+|\zeta_1|^2)z^2\end{equation}
$$+\left((-\overline{\zeta}_3)\prod_{j=1}^2(1+|\zeta_j|^2)+
(-\overline{\zeta}_2)(1+|\zeta_1|^2)(-\overline{\zeta}_2\zeta_1)\right)z^3$$
$$+\left((-\overline{\zeta}_4)\prod_{j=1}^3(1+|\zeta_j|^2)
+(-\overline{\zeta}_3)\prod_{j=1}^2(1+|\zeta_j|^2)(-\overline{\zeta}_3\zeta_2-2\overline{\zeta}_2\zeta_1)
+(-\overline{\zeta}_2)(1+|\zeta_1|^2)(\overline{\zeta}_2^2\zeta_1^2)\right)z^4$$
$+...$. The general form of this expansion is discussed in an appendix.

The fact that these algebraic correspondences continuously extend to analytic correspondences
depends on the following Plancherel-esque formulas (which explain the interest in root subgroup coordinates).
For $k_i$ as in Theorem \ref{introtheorem1},
\begin{equation}\label{Planch1}det(A(k_1)^*A(k_1))=det(1-C(k_1)^*C(k_1))
=det(1+\dot B(y)^*\dot B(y))^{-1}=\prod_{i\ge 1}(1+\vert \eta_i
\vert^2)^{-i}\end{equation} and
\begin{equation}\label{Planch2}det(A(k_2)^*A(k_2))=det(1-C(k_2)^*C(k_2)) =det(1+\dot
B(x)^*\dot B(x))^{-1}=\prod_{k\ge 1}(1+\vert \zeta_k\vert^2)^{-k}\end{equation}
where in the third
expressions, $x$ and $y$ are viewed as multiplication operators on
$H=L^2(S^1)$, with Hardy space polarization. In (\ref{Planch1}), the first two terms are nonzero iff
$k_1\in W^{1/2}$, the third is nonzero iff $y \in W^{1/2}$, and the third is nonzero iff $\eta\in \mathbf w^{1/2}$.

Finally we need to explain why the limits in I.2 and II.2 are $W^{1/2}$ limits, as opposed to simply pointwise (a.e.) limits. The basic fact is that $det(A(g)A(g^{-1}))=det(1-B(g)B(g)^*)$ is a continuous positive definite function on $W^{1/2}(S^1,SU(2))$ which determines the topology of this group. The positive definite
function associated to the vacuum vector $v_0$ for the so called basic representation of the Kac-Moody central extension of $ W^{1/2}(S^1,SU(2))$ is the section $det(A)$, viewed as a function on the central extension (see chapter 10 of \cite{PS}; this reference emphasizes that this is true for the universal central extension of
the Hilbert-Schmidt restricted group $U_{\mathcal L_2}$, but this implies our assertion because of Proposition \ref{opertopology}). For the basic representation tensored with its dual, the positive definite function associated to
the vector $v_0\otimes \overline{v_0}$ is the scalar function $det(A(g)A(g^{-1}))$.
The continuity of this function is equivalent to the strong operator continuity of the corresponding unitary representation, and this strong operator notion of convergence is equivalent to convergence in $W^{1/2}(S^1,SU(2))$
(In \cite{PS} this assertion is proven more universally for the Hilbert-Schmidt restricted group $U_{\mathcal L_2}$, and this implies our assertion because of Proposition \ref{opertopology}; see chapter 13 of \cite{Dixmier} for background on positive definite functions on groups).
\end{proofidea}

\begin{question}\label{W1/2projectivemap} Given $k_2$ as in the theorem, one obtains a based holomorphic map
$$ (\Delta,0) \to (\mathbb C\mathbb P^1, [0:1]): z \to [c_2(z):d_2(z)] $$
which has radial boundary values. How exactly does one describe the boundary conditions,
and how does one recover $k_2$ from this map? This question comes up later in the paper.
\end{question}

\begin{theorem}\label{introtheorem2} Suppose $g\in W^{1/2}(S^1,SU(2))$. The following are
equivalent:

(i) The (block) Toeplitz operator $A(g)$ and shifted Toeplitz operator $A_1(g)$
are invertible.

(ii) $g$ has a triangular factorization $g=lmau$.

(iii) $g$ has a (root subgroup) factorization of the form
$$g(z)=k_1^*(z)\left(\begin{matrix} e^{\chi(z)}&0\\
0&e^{-\chi(z)}\end{matrix}\right)k_2(z)$$ where $k_1$ and $k_2$
are as in Theorem \ref{introtheorem1} and $\chi \in W^{1/2}(S^1,i\mathbb R)$.

\end{theorem}

\begin{proofidea} The equivalence of (i) and (ii) is standard (see also (\ref{Toeplitzdiag}) below).

Suppose that $g\in L_{fin}SU(2)$. If $g$ has a root subgroup factorization as in (iii),
one can directly find the triangular factorization (see Proposition \ref{trifactorization} below), and from this explicit expression,
one can see how to recover the factors $\eta,\chi,\zeta$ (Incidentally, $\eta$ and $\zeta$ have finitely many nonzero terms, but this is not so for $\chi$, hence this calculation is not purely algebraic).

As was the case for Theorem \ref{introtheorem1}, the fact that these correspondences extend to analytic correspondences
depends on a number of Plancherel-esque identities.
For $g\in W^{1/2}(S^1,SU(2))$ satisfying the conditions in Theorem \ref{introtheorem2},

\begin{equation}\label{Toeplitz0}det(A(g)^*A(g))=
\left(\prod_{i=0}^{\infty}\frac{1}{(1+\vert\eta_i\vert^2)^{i}}\right)\times
\left(\prod_{
j=1}^{\infty}e^{-2j\vert\mathbf{\chi}_j\vert^2}\right)\times
\left(\prod_{k=1}^{\infty}\frac
{1}{(1+\vert\zeta_k\vert^2)^{k}}\right)\end{equation}

\begin{equation}\label{Toeplitz1}det(A_1(g)^*A_1(g))=
\left(\prod_{i=0}^{\infty}\frac{1}{(1+\vert\eta_i\vert^2)^{i+1}}\right)\times
\left(\prod_{
j=1}^{\infty}e^{-2j\vert\mathbf{\chi}_j\vert^2}\right)\times
\left(\prod_{k=1}^{\infty}\frac
{1}{(1+\vert\zeta_k\vert^2)^{k-1}}\right)\end{equation}
(where $A_1$ is the shifted Toeplitz operator)
\begin{equation}\label{Toeplitzdiag}a_0(g)^2=\frac{det(A_1(g)^*A_1(g)) }{det(A(g)^*A(g))}
=\left(\prod_{i=0}^{\infty}\frac{1}{(1+\vert\eta_i\vert^2)}\right)\times
\left(\prod_{k=1}^{\infty}(1+\vert\zeta_k\vert^2)\right)\end{equation}

\end{proofidea}

Note that because $g$ is unitary, i.e. $g^{-1}=g^*$ on $S^1$, parts (i) and (ii) are obviously inversion invariant, and this does not depend on the hypothesis that $g\in W^{1/2}$: if $g:S^1 \to SU(2)$ has the triangular factorization $g=lmau$, then $g^{-1}=g^*$ has triangular factorization $g^{-1}=u^* m^* a l^*$. On the other hand part (iii), the existence of a root subgroup factorization, is not obviously inversion invariant.

\begin{corollary}Suppose $g\in W^{1/2}(S^1,SU(2))$. Then $g$ has a root subgroup factorization (as in (iii)
of Theorem \ref{introtheorem2}) if and only if
$g^{-1}$ has a root subgroup factorization.
\end{corollary}

We have used the hypothesis that $g\in W^{1/2}$ so that we can use the identities (\ref{Toeplitz0}) and (\ref{Toeplitz1}) to prove that the existence of a root subgroup factorization implies invertibility of the Toeplitz determinants. A central question related to the generalizations in the following sections is whether the hypothesis $g\in W^{1/2}$ is crucial for inversion invariance of root subgroup factorization.

\subsection{Coordinates for $W^{1/2}(S^1,SU(2))$}\label{W1/2II}

Theorem \ref{introtheorem2} implies the following

\begin{corollary} $W^{1/2}(S^1,SU(2))$ is a topological Hilbert manifold
modeled on the root subgroup parameters $\{((\eta_i)_{i\ge 0},(\chi_j)_{j\ge 1},(\zeta_k)_{k\ge 1})\in l^2\times l^2 \times l^2\}\times \{e^{\chi_0}\in S^1\}$ for the open set of loops with invertible $A$ and $A_1$.
\end{corollary}

As we noted in the introduction, it is not possible to use this (or any) coordinate to define a smooth structure which is translation invariant (because $W^{1/2}(S^1,su(2))$ is not a Lie algebra).

There are other coordinates, and this will be important when we consider VMO loops, because we will not be able to characterize VMO loops in terms of the coordinates $\eta$ and $\zeta$.

\begin{theorem}\label{xycoordinate} (a) The maps
$$\{k_1=\left(\begin{matrix}a_1&b_1\\-b_1^*&a_1^*\end{matrix}\right) \text{as in I.1-3 of Theorem \ref{introtheorem1} }\} \to \{y=\sum_{n=0}^{\infty}y_nz^n\in W^{1/2}(S^1)\}: k_1\ \to y $$
and
$$\{k_2=\left(\begin{matrix}d_2^*&-c_2^*\\c_2&d_2\end{matrix}\right) \text{as in II.1-3 of Theorem \ref{introtheorem1} }\} \to \{x=\sum_{n=1}^{\infty}x_nz^n\in W^{1/2}(S^1)\}: k_2\ \to x $$
are bijections.\\

(b)  $(y,\chi,x)$ is a topological coordinate system for the open set of loops in $W^{1/2}(S^1,SU(2))$ with invertible $A$ and $A_1$.
\end{theorem}

The change of coordinates from $\zeta$ to $x$ is discussed in an Appendix.

\begin{proof} In the first part of the proof, we will prove a more general result for measurable
loops, which we will exploit in the next section.

For part (a) we will use the Grassmannian model for the measurable loop group $Meas(S^1,U(2))$, see Proposition (8.12.4) of \cite{PS}, which describes the $Meas(S^1,U(2))$ orbit of $H_+$ in the Grassmannian
of $H=L^2(S^1,\mathbb C^2)$ (see (\ref{polarization})).
Given $x(z)=\sum_{n=1}^{\infty}x_nz^n\in L^2(S^1)$, let $W$ denote the smallest closed $M_z$-invariant subspace
containing the vectors $\left(\begin{matrix}1\\0\end{matrix}\right)$ and $\left(\begin{matrix}x^*\\1\end{matrix}\right)$. We claim that
\begin{equation}\label{grassconditions}\bigcap_{k\ge 0} z^kW=0 \text{  and  }  \bigcup_{k\le 0} z^{-k}W \text{ is dense in } H \end{equation}
For the first condition, suppose that $v$ is a point in the intersection. For each $N>0$ it is possible to
write $v(z)=\left(\begin{matrix}z^Nf_N(z)+z^Ng_N(z)x^*(z)\\z^Ng_N(z)\end{matrix}\right)$, where $f_N,g_N\in \dot H_+$
The second component of $v$ has to be identically zero. This implies $g_N$ has to be zero. Now the first component of
$v$ also has to vanish. The second condition is equivalent to showing that the subspace spanned by $\left(\begin{matrix}s(z)+t(z)x^*(z)\\t(z)\end{matrix}\right)$, where $s$ and $t$ are finite Fourier series, is dense
in $L^2(S^1)$. This is obvious.

This implies that $W$ is in the Grassmannian in Proposition (8.12.4) of \cite{PS}, and hence there exists
$k_2\in Meas(S^1,U(2))$ such that $k_2H_+=W$ ($k_2$ is obtained by taking an orthonormal basis for the two dimensional orthogonal complement of $zW$ inside $W$, a Gram-Schmidt type process). This implies that $k_2^{-1}W=H_+$, hence $k_2^{-1}\left(\begin{matrix}1&x^*\\0&1\end{matrix}\right)$ is holomorphic in the disk, and hence
$$k_2(z)=\lambda(z)\left(\begin{matrix}d_2^*(z)&-c_2^*(z)\\c_2(z)&d_2(z)\end{matrix}\right)= \left(\begin{matrix} 1&x^*(z)\\
0&1\end{matrix} \right)\left(\begin{matrix}\mathbf a_2&0\\
0&\mathbf a_2^{-1}\end{matrix} \right)\left(\begin{matrix} \alpha_2 (z)&\beta_2 (z)\\
\gamma_2 (z)&\delta_2 (z)\end{matrix} \right)$$
where $\mathbf a_2>0$, $\lambda^2=det(k_2):S^1 \to S^1$, $|c_2|^2+|d_2|^2=1$ on $S^1$.
From the second row of this equality, we see that $\lambda$ extends to a holomorphic function in $\Delta$.
$\lambda$ cannot vanish because $\gamma_2$ and $\delta_2$ cannot simultaneously vanish. Thus $\lambda$
is a constant; the normalizations in II.1-3 force $\lambda=1$.

We now consider the hypothesis in part (a) of the theorem, i.e. $x\in W^{1/2}$. This implies that
$$det(A(k_2)A(k_2^{-1}))=det(1-B(k_2)B(k_2)^*)=det(1+\dot B(x)\dot B(x)^*)^{-1}$$ is positive.
Therefore $k_2\in W^{1/2}$. The claim about $k_1$ and $y$ is similar.

Part (b) follows from (a).

\end{proof}

The preceding proof is abstract. In the next section (see Lemma \ref{inverse}) we will show how to solve for the
unitary loop corresponding to a given $x=\sum_{n=1}^{\infty}x_nz^n\in L^2(S^1)$.
Here we will simply state the result, which has a transparent meaning when $x\in W^{1/2}$.

\begin{theorem}\label{keyidentities1} Given $x=\sum_{n=1}^{\infty}x_nz^n\in W^{1/2}(S^1)$, the
corresponding loop $k_2\in W^{1/2}(S^1,SU(2))$ is determined by the identities
$$\mathbf a_2^2=\frac{1}{\langle 1|(1+\dot B(x)\dot B(x)^*)^{-1}|1\rangle}$$
$$\gamma_2^*=-\mathbf a_2^2(1+\dot B(x)^*\dot B(x))^{-1}(x^*) $$
and
$$\delta_2=\mathbf a_2^2(1+\dot B(x)\dot B(x)^*)^{-1}(1)  $$

\end{theorem}

\begin{remark} One can ask, can one dispense with root subgroup coordinates altogether and use $g_-$ as a
coordinate for the top stratum of $g\in W^{1/2}(S^1,SU(2))/SU(2)$? This is unknown to us
and discussed further in an appendix.\end{remark}

\subsection{Lower Strata}

In this subsection we will briefly outline how the theory above generalizes to lower strata, i.e.
non-generic loops. This involves a finite codimensional conditioning of what we have already done.

Let $N^+$ denote the group of (Lebesgue equivalence classes of) loops $u:S^1 \to SL(2,\mathbb C)$ which
are almost sure radial boundary values of maps $u$ as in \ref{factorization}, i.e. $u:\Delta \to SL(2,\mathbb C)$ is holomorphic and $u(0)$ is unipotent upper triangular. $N^-$ is similarly defined with $l(\infty)$ unipotent lower triangular. Suppose that
\begin{equation}\label{weylelt}w=\left(\begin{matrix}0&1\\-1&0\end{matrix}\right)^{\epsilon}
\left(\begin{matrix}z^n&0\\0&z^{-n}\end{matrix}\right) \end{equation} where $\epsilon=0$ or $1$. We say that $g\in L^1(S^1,SL(2,\mathbb C))$ belongs to the $w$ stratum if $g$ has a generalized triangular factorization of the form
\begin{equation}\label{genfactorization} g=lwmau,\end{equation}
where $l,m,a,u$ are as in \ref{factorization}. When $w\ne 1$, the $l$ and $u$ factors are not unique; to obtain uniqueness, these factors must be conditioned.

\begin{lemma}\label{genfactlemma} (a) Each $u\in N^+$ has a unique decomposition $u=u_-u_+$ where
$u_-\in N^{+}\cap wN^{-}w^{-1}$ and $u_+\in N^{+}\cap wN^{+}w^{-1}$ (or in the opposite order),
and there is a similar decomposition for $N^-$.

(b) $g$ belongs to the $w$ stratum iff $w^{-1}g$ or $gw^{-1}$ has a standard triangular factorization.
\end{lemma}

\begin{proof} (a) Suppose that $n>0$, $\epsilon=0$, and $u\in N^+$. We calculate
$$w^{-1}uw=\left(\begin{matrix}u_{11}&z^{-2n}u_{12}\\z^{2n}u_{12}&u_{22}\end{matrix}\right) $$
This is in $N^-$ iff $u_{11}=1$, $u_{12}$ is a polynomial of degree $<2n$, $u_{21}=0$, and $u_{22}=1$
Thus
$$N^+\cap wN^-w^{-1}=\{u=w^{-1}\left(\begin{matrix}1&x^{*(2n)} \\0 &1 \end{matrix}\right)w\} \text{ where } x^{*(2n)}=\sum_{k=1}^{2n}x_k^* z^{-k} $$
and
$$N^+\cap wN^+w^{-1}=\{u: (u_{12})_k=0, k<2n\} $$
Now note
$$\left(\begin{matrix}u_{11}&u_{12}\\u_{12}&u_{22}\end{matrix}\right)\left(\begin{matrix}1&-z^{-2n}x^{*(2n)} \\0 &1 \end{matrix}\right)= \left(\begin{matrix}u_{11}&-u_{11}z^{-2n}x^{*(2n)}+u_{12}\\u_{12}&-u_{21}z^{-2n}x^{*(2n)}+u_{22}\end{matrix}\right)$$
Since $u_{11}(0)=1$, we can uniquely solve for $x^{*(2n)}$ such that the 1,2 entry is $o(z^{2n})$. This implies the unique decomposition in (a). The other cases, $n<0, \epsilon=0$ and so on, are done in a similar way.

(b) Given $w$ and $g=lwmau$, by part (a) we can choose $l$ to be in $N^-\cap wN^-w^{-1}$. Then $g=wl'mau$, where $l'=w^{-1}lw\in N^-$, hence
$w^{-1}g$ has a standard triangular decomposition. The other assertions are similar.
\end{proof}

Note that in the case $n>0$ and $\epsilon=0$, the above proof shows that the subgroup
$N^+\cap wN^-w^{-1}$ can be parameterized by
$x_1,...,x_{2n}$ or (real analytically) by $\zeta_1,...,\zeta_{2n}$. To obtain uniqueness in (\ref{genfactorization}), we will set these parameters to zero and consider the complementary parameters.

Let $\Sigma_w$ denote the set of $g\in W^{1/2}(S^1,SU(2))$ which belong to the $w$ stratum. The following is
not new. In chapter 8 of \cite{PS} there is a geometric interpretation attached to $w$ triangular factorization
(Because we are considering triangular factorization as opposed to Riemann-Hilbert factorization, we should consider the flag space in section 8.7 of \cite{PS} rather than the Grassmannian model; the text refers to smooth loops, but the authors were well aware that the theory applies to $W^{1/2}$ loops, see page 84 of \cite{PS}).

\begin{proposition}\label{PSproposition} (a) $W^{1/2}(S^1,SU(2))$ is the disjoint union of the $\Sigma_w$, as $w$ varies.

(b) The factors $w,m$ and $a$ in (\ref{genfactorization}) are uniquely determined by $g\in \Sigma_w$.
If we require that $l\in N^-\cap wN^-w^{-1}$, then $l$ and $u$ are uniquely determined, and similarly
if we require that $u\in N^+\cap wN^+w^{-1}$, then $l$ and $u$ are uniquely determined.

\end{proposition}

Our main point is to show that one can use root subgroup factorization to parameterize these lower strata.
From an analytic perspective the results may seem ad hoc. In Lie theoretic terms, here is what is going on: each $k_i$ is a product of a certain collection of real root subgroup factors, and $\left(\begin{matrix}e^{\chi}&0\\0&e^{-\chi}\end{matrix}\right)$ corresponds to imaginary root factors. In turn $w$ has a factorization (which we have not made explicit) in terms of reflections corresponding to real roots. The factorization of $w$ determines which root subgroup factors have to be 'turned off' to parameterize the corresponding stratum.

\begin{theorem}\label{W12lowerstrata} Suppose $w$ is written as in \ref{genfactorization} and
$$g(z)=k_1(\eta)^*(z) w \left(\begin{matrix} e^{\chi(z)}&0\\
0&e^{-\chi(z)}\end{matrix}\right)k_2(\eta)(z)$$ where $k_1$ and $k_2$
are as in Theorem \ref{introtheorem1} and $\chi \in W^{1/2}(S^1,i\mathbb R)$.

(a) Suppose $n>0$ and $\epsilon=0$. If $\zeta_k=0$ for $k=1,...,2n$, or equivalently $x_k=0$
for $k=1,...,2n$, then
$g\in\Sigma_w$. Conversely $g\in\Sigma_w$ has a unique expression of this form.

(b) Suppose that $n<0$ and $\epsilon=0$. If $\eta_i=0$ for $i=0,...,2|n|-1$, then
$g\in\Sigma_w$. Conversely $g\in\Sigma_w$ has a unique expression of this form.

(c) Suppose $n>0$ and $\epsilon=1$. If $\zeta_k=0$ for $j=1,...,2n+1$, then
$g\in\Sigma_w$. Conversely $g\in\Sigma_w$ has a unique expression of this form.

(d) Suppose that $n\le 0$ and $\epsilon=1$. If $\eta_i=0$ for $i=0,...,2|n-1|$, then
$g\in\Sigma_w$. Conversely $g\in\Sigma_w$ has a unique expression of this form.
\end{theorem}

\begin{proof} Consider part (a). We must show that $gw^{-1}$ has a standard triangular factorization.
Note that $w$ commutes with $\left(\begin{matrix} e^{\chi(z)}&0\\
0&e^{-\chi(z)}\end{matrix}\right)$.
$$wk_2w^{-1}= \left(\begin{matrix} 1&z^{2n}x^{*}\\
0&1\end{matrix} \right)\left(\begin{matrix}\mathbf a_2&0
\\0&\mathbf a_2^{-1}\end{matrix}\right)\left(\begin{matrix} \alpha_2&z^{2n}\beta_2\\
z^{-2n}\gamma_2&\delta_2\end{matrix} \right)$$
Because of the conditions on $\zeta$ (or $x$), on the right hand side the third matrix is in $N^+$. A calculation (see the proof of Proposition \ref{trifactorization} below for the case $w=1$, which is easily modified) now implies
$gw^{-1}(z)=k^*_1(z)\left(\begin{matrix} e^{\chi(z)}&0\\
0&e^{-\chi(z)}\end{matrix}\right)(wk_2(\eta)w^{-1})(z)$ has triangular factorization $gw^{-1}=lmau$ where
$$l=\left(\begin{matrix} \alpha_1^{*}&-\gamma_1^*\\
\beta_1^{*}&\delta_1\end{matrix}
\right)\left(\begin{matrix}e^{-\chi_+^*}&0
\\0&e^{\chi_+^*}\end{matrix}\right)
\left(\begin{matrix}1&M_-\\0&1\end{matrix}\right)
$$

$$\label{ma}m=\left(\begin{matrix}e^{\chi_0}&0\\0& e^{-\chi_0}\end{matrix}\right),
\quad a=\left(\begin{matrix}a_0&0\\0&a_0^{-1}\end{matrix}\right)=
\left(\begin{matrix}a_1a_2&0\\0&(a_1a_2)^{-1}\end{matrix}\right)$$

$$u=\left(\begin{matrix} 1&M_{0+}\\
0&1\end{matrix} \right)\left(\begin{matrix}e^{\chi_+}&0
\\0&e^{-\chi_+}\end{matrix}\right)\left(\begin{matrix} \alpha_2&z^{2n}\beta_2\\
z^{-2n}\gamma_2&\delta_2\end{matrix}\right) $$
and
$$M=(a_0m_0)^{-2}e^{2\chi_+^*}Y+e^{2\chi_+}z^{2n}X^{*}$$
$Y=\mathbf a_1^{2},X=\mathbf a_2^{-2}x$.

The other cases are similar.

\end{proof}

We have now shown that $W^{1/2}(S^1,SU(2))$ has a stratification which restricts to
the well-known stratification for smooth loops. As explained in the proof of (8.6.6) of \cite{PS},
this implies the following

\begin{corollary}\label{W12homotopy}The inclusions $L_{fin}SU(2) \to C^{\infty}(S^1,SU(2)) \to W^{1/2}(S^1,SU(2))$ are homotopy equivalences.
\end{corollary}

\section{The $L^2$ Theory}\label{L2}

We now ask whether there are $L^2$ analogues of Theorems \ref{introtheorem1} and \ref{introtheorem2}.
Here is a naive $L^2$ analogue of Theorem \ref{introtheorem1} (we consider just the second set of equivalences):

\begin{question}\label{introtheorem3}  Suppose that $k_2:S^1 \to SU(2)$ is Lebesgue measurable. Are the following equivalent:

(II.1) $k_2$ has the form
$$k_2(z)=\left(\begin{matrix} d_2^{*}(z)&-c_2^{*}(z)\\
c_2(z)&d_2(z)\end{matrix} \right),\quad z\in S^1,$$ where $c_2,d_2\in
H^0(\Delta)$ do not simultaneously vanish, $c_2(0)=0$ and $d_2(0)>0$.

(II.2) There exists a unique $(\zeta_k)\in l^2$ such that
$$k_2(z)=\lim_{n\to\infty}\mathbf a(\zeta_n)\left(\begin{matrix} 1&\zeta_nz^{-n}\\
-\bar{\zeta}_nz^n&1\end{matrix} \right)..\mathbf
a(\zeta_1)\left(\begin{matrix} 1&
\zeta_1z^{-1}\\
-\bar{\zeta}_1z&1\end{matrix} \right)$$ where the limit is understood in terms of convergence in measure.

(II.3) $k_2$ has triangular factorization of the form
$$\left(\begin{matrix} 1&\sum_{j=1}^{\infty}x^*_jz^{-j}\\
0&1\end{matrix} \right)\left(\begin{matrix}\mathbf a_2&0\\
0&\mathbf a_2^{-1}\end{matrix} \right)\left(\begin{matrix} \alpha_2 (z)&\beta_2 (z)\\
\gamma_2 (z)&\delta_2 (z)\end{matrix} \right)$$
where $\mathbf a_2>0$.

For $k_2$ satisfying these conditions, we will see that
$$\mathbf a_2^2=d_2(0)^{-2}=\prod_{k=1}^{\infty}(1+|\zeta_k|^2)=|\gamma_2|^2+|\delta_2|^2 \qquad(\text{on }S^1)$$
\begin{equation}\label{2ndformula}=1+\langle x|(1+ B(z^{-1}x)B(z^{-1}x)^*)^{-1}x\rangle_{L^2}=
\frac{1}{\langle1|(1+\dot B(x)\dot B(x)^*)^{-1}1\rangle_{L^2} }\end{equation}
(the meaning of the operators is explained in Lemma \ref{inverse})
and
$$|\alpha_2|^2+|\beta_2|^2=\mathbf a_2^{-2}(1+|x|^2) $$
on $S^1$.

\end{question}

In the first part of this section, our goal is to explain how the various implications have to be qualified.
One complication in this general context is the existence of singular inner functions (see page 370 of \cite{Rudin}).

\begin{example}\label{non-example}A simple non-example to bear in mind for (II.1) is
$$k_2(z)=\left(\begin{matrix} d_2^*(z)&0\\
0&d_2(z)\end{matrix} \right) \text{ where } d_2=\frac{z-t}{1-tz} $$
and $0<t<1$. This does not satisfy the hypothesis
that $c_2$ and $d_2$ are simultaneously nonvanishing, which is critical to show that the Toeplitz operator
$A(k_2)$ is injective.

A complex example for (II.1) is a $k_2$
where $c_2(z)=\sqrt{t_1}C_2(z)$, $d_2(z)=\sqrt{t_2}D_2(z)$, $C_2$ and $D_2$ are inner functions which do not simultaneously vanish in $\Delta$, and $t_1,t_2>0$, $t_1+t_2=1$.

\end{example}

It is obvious that (II.3) implies (II.1). The important point is that the triangular factorization implies that $c_2$ and $d_2$ do not simultaneously vanish in $\Delta$. For later use, notice that (II.3) and the special unitarity of $ k_2$ imply ("the unitarity equations")
\begin{equation}\label{1st2ndeqn1}\mathbf a_2\alpha_2+x^*\mathbf a_2^{-1}\gamma_2=\mathbf a_2^{-1}\delta_2^*, \quad
\mathbf a_2\beta_2+x^*\mathbf a_2^{-1}\delta_2=-\mathbf a_2^{-1}\gamma_2^*\end{equation}
and
\begin{equation}\label{3rdeqn1}\mathbf a_2^{-2}(\gamma_2^*\gamma_2+\delta_2^*\delta_2)=1\end{equation}
These equations imply
\begin{equation}\label{4theqn}\alpha_2=-\mathbf a_2^{-2}x^*\gamma_2+\mathbf a_2^{-2}\delta_2^*\quad
\text{and}\quad \beta_2=-\mathbf a_2^{-2}x^*\delta_2-\mathbf a_2^{-2}\gamma_2^*\end{equation}
Applying the $(\cdot)_{0+}$ projection to each of these,  we obtain
$\alpha_2=1-(X^*\gamma_2)_+$ and $\beta_2=-(X^*\delta_2)_{0+}$.
Using (\ref{4theqn}) again, on $S^1$
$$|\alpha_2|^2+|\beta_2|^2=\mathbf a_2^{-4}((-x^*\gamma_2+\delta_2^*)(-x\gamma_2^*+\delta_2)+(x^*\delta_2+\gamma_2^*)(x\delta_2^*+\gamma_2)) $$
Expand this and use the obvious cancelations. Together with (\ref{3rdeqn1}), this implies
\begin{equation}\label{normidentity}|\alpha_2|^2+|\beta_2|^2=\mathbf a_2^{-2}(1+|x|^2)\end{equation}
as claimed in the last part of Question \ref{introtheorem3}.

Now assume (II.1). We can determine $\zeta_1,\zeta_2,...$ using the Taylor series (\ref{Taylorexp}) for $c_2/d_2$ (note this is not identically zero, unlike the first loop in Example \ref{non-example}). Let
$$\left(\begin{matrix} d_2^{(n)*}(z)&-c_2^{(n)*}(z)\\
c_2^{(n)}(z)&d_2^{(n)}(z)\end{matrix} \right)=\mathbf a(\zeta_n)\left(\begin{matrix} 1&\zeta_nz^{-n}\\
-\bar{\zeta}_nz^n&1\end{matrix} \right)..\mathbf
a(\zeta_1)\left(\begin{matrix} 1&
\zeta_1z^{-1}\\
-\bar{\zeta}_1z&1\end{matrix} \right)$$ Because the polynomials $c_2^{(n)}(z)$ and $d_2^{(n)}(z)$
are bounded by $1$ in the disk, given any subsequence, there exists a subsequence for which this pair will converge
uniformly on compact subsets of $\Delta$. The limits, denoted
$\widetilde {c_2}(z)$ and $\widetilde{d_2}(z)$, are bounded by $1$, hence will have radial boundary values.
We will use the following elementary fact repeatedly.

\begin{lemma}\label{subseqlemma} Suppose that $f_n\in L^{\infty}H^0(\Delta)$ and $f_n$ converges uniformly on compact subsets to $f\in L^{\infty}H^0(\Delta)$. Then there exists a subsequence $f_{n_j}$ which converges pointwise
a.e. on $S^1$ to $f$.
\end{lemma}

\begin{proof} Because each $f_j$ and $f$ are essentially bounded, each $f_j$ and $f$ has radial limits, on a common
subset $E$ of $S^1$ of full Lebesgue measure. For each
$j$ there exists $n_j$ such that $|f_{n_j}-f|<\frac 1j$ on $(1-\frac{1}{j})S^1$. The subsequence
$f_{n_j}$ then converges pointwise on $E$ to $f$.
\end{proof}

It follows that for some subsequence,
$$\widetilde{k_2}(\zeta)(z):=\lim_{j\to\infty}\left(\begin{matrix} d_2^{(n_j)*}(z)&-c_2^{(n_j)*}(z)\\
c_2^{(n_j)}(z)&d_2^{(n_j)}(z)\end{matrix} \right)
$$
exists in the pointwise Lebesgue a.e. sense on the circle.
Furthermore the sequence of zetas corresponding to $\widetilde{k_2}$ is $\zeta_1,...$. Therefore
using (\ref{Taylorexp}) $c_2/d_2= \widetilde{c_2}/\widetilde{d_2}$. Together with unitarity and the simultaneous nonvanishing
condition on $c_2,d_2$, this implies
$$\lambda:=\frac{\widetilde {c_2}}{c_2}=\frac{\widetilde {d_2}}{d_2} $$
is a holomorphic function in $\Delta$ with radial boundary values and $|\lambda|=1$ on $S^1$. Such a function has a unique factorization $\lambda=\lambda_b\lambda_s$, where $\lambda_b$ is a Blaschke product and $\lambda_s$ is a singular inner function, i.e.
\begin{equation}\label{Caratheodory}\lambda_s(z)=exp(\int_{S^1}\frac{z+e^{i\theta}}{z-e^{i\theta}}d\nu(\theta))\end{equation}
where $\nu$ is a finite positive measure which is singular with respect to Lebesgue measure (see page 370 of \cite{Rudin}). The integral, as a holomorphic function of $z$ is (up to a constant) usually referred to as the Caratheodory function of $\nu$; because $\nu$ is singular, the Caratheodory function is not $W^{1/2}$, hence is forced to vanish when $k_2$ is $W^{1/2}$ (or more generally VMO). The simultaneous nonvanishing condition
implies that $\lambda_b=1$. Since $\widetilde{d_2}(0),d_2(0)>0$, $\lambda(0)=1$, and $d_2(0)=\prod_{k> 0}\mathbf a(\zeta_k)=\prod_{k> 0}(1+|\zeta_k|^2)^{-1/2}>0$.
It follows
that $\zeta\in l^2$. This implies the following

\begin{theorem}\label{rootsubgpfact} Assume (II.1) in Question \ref{introtheorem3}. Then there exists a unique $(\zeta_k)\in l^2$
and a singular inner function $\lambda$ with $\lambda(0)=1$ such that
$$k_2(z)=\left(\begin{matrix} \lambda(z)&0\\
0&\lambda^{-1}(z)\end{matrix} \right)\lim_{n\to\infty}\mathbf a(\zeta_n)\left(\begin{matrix} 1&\zeta_nz^{-n}\\
-\bar{\zeta}_nz^n&1\end{matrix} \right)..\mathbf
a(\zeta_1)\left(\begin{matrix} 1&
\zeta_1z^{-1}\\
-\bar{\zeta}_1z&1\end{matrix} \right)$$ where the limit is understood in terms of convergence in measure.
\end{theorem}

\begin{question} Is the map from a $k_2$ as in (II.1) to a based holomorphic map
$\Delta \to \mathbb C\mathbb P^1$, with a radial extension to the boundary, surjective? Because of the existence of singular inner functions, it is far from injective.
\end{question}

Now assume $\zeta\in l^2$ as in (II.2). We will show that this implies (II.1), sans the simultaneous nonvanishing condition, and we will explain why we do not necessarily obtain a factorization as in (II.3).  Note we are free to use the unitarity equations for sufficiently regular $\zeta$, e.g. $\zeta\in \mathbf w^{1/2}$. In the course of the argument, we will also prove (\ref{2ndformula}), among other formulas.

We basically proved the following in \cite{Pi1}, but missed one elementary point at the very end of the argument.

\begin{proposition}\label{prequellemma}
Suppose that $\zeta=(\zeta_n)\in l^2$. Let
$$k_2^{(N)}=\left(\begin{matrix} d^{(N)*}&-c^{(N)*}\\
c^{(N)}&d^{(N)}\end{matrix} \right):=
\left(\prod_{n=1}^{N}\mathbf a(\zeta_n)\right)\left(\begin{matrix} 1&\zeta_Nz^{-N}\\
-\bar{\zeta}_Nz^N&1\end{matrix} \right)..\left(\begin{matrix} 1&\zeta_1z^{-1}\\
-\bar{\zeta}_1z&1\end{matrix} \right)$$ Then $c^{(N)}$ and
$d^{(N)}$ converge uniformly on compact subsets of $\Delta$ to
holomorphic functions $c=c(\zeta)$ and $d=d(\zeta)$, respectively,
as $N\to\infty$. The functions $c$ and $d$ have radial limits at
a.e. point of $S^1$, $c$ and $d$ are uniquely determined by these
radial limits,
$$k_2(z)=k_2(\zeta)(z):=\left(\begin{matrix} d(\zeta)^*(z)&-c(\zeta)^*(z)\\
c(\zeta)(z)&d(\zeta)(z)\end{matrix} \right)\in
Meas(S^1,SU(2,\mathbb C))$$
\end{proposition}

Note that if $\zeta\in l^1$, then the product actually converges absolutely around the circle.
So one subtlety here is relaxing summability to square summability. Note also that the proof that (II.1) implies (II.2) shows that there exist convergence in measure limit points. So the second subtlety
is showing that there is a unique limit point.

\begin{proof} For any sequence $\zeta$ (not necessarily $l^2$), $c^{(N)}$ and $d^{(N)}$ will have subsequences
which converge uniformly on compact subsets of $\Delta$. We claim that if $\zeta\in l^2$, then these limits
are unique. The fact that $\zeta \in l^2$
implies that the product of the $\mathbf a(\zeta_k)$ converges. If
\begin{equation}\left(\begin{matrix}\alpha_2^*(z)&\beta_2^*(z)\\
\gamma_2(z)&\delta_2(z)\end{matrix}\right)=\lim_{n\to\infty}\left(\begin{matrix} 1&\zeta^-_nz^{-n}\\
\zeta^+_nz^n&1\end{matrix} \right)...\left(\begin{matrix} 1&
\zeta^-_1z^{-1}\\
\zeta^+_1z&1\end{matrix} \right),\end{equation}
then there are explicit formulas
$$\gamma_2(z)=\sum_{n=1}^{\infty}\gamma_{2,n}z^n,$$
$$\gamma_{2,n}=\sum (\zeta^+_{i_1})\zeta^-_{j_1}...(\zeta^+_{
i_r})\zeta^-_{j_r}(\zeta^+_{i_{r+1}}),$$ where the sum is over
multiindices satisfying
$$0<i_1<j_1<..<j_r<i_{r+1},\quad\sum i_{*}-\sum j_{*}=n,$$
and
$$\delta_2(z)=1+\sum_{n=1}^{\infty}\delta_{2,n}z^n,$$
$$\delta_{2,n}=\sum\zeta^-_{j_1}(\zeta^+_{i_1})...\zeta^-_{j_r}(
\zeta^+_{i_r}),$$ where the sum is over multiindices satisfying
$$0<j_1<i_1<..<i_r,\quad\sum (i_{*}-j_{*})=n$$
To see that these expressions
have well-defined limits consider
the formula for the $n$th coefficient of $\delta_2$, and let
$\mathcal P(n)$ denote the set of partitions of $n$ (i.e.
decreasing sequences $n_1\ge n_2 \ge ..\ge n_l>0$, where $\sum
n_j=n$ is the magnitude and $l=l({n_j})$ is the length of the
partition). Then
\begin{equation}\label{coefficient}\vert \delta_{2,n}\vert
\le\sum\vert \zeta_{i_1}\vert \vert \bar{\zeta}_{j_1}\vert
...\vert \zeta_{i_r}\vert \vert \bar{\zeta}_{j_r}\vert
,\end{equation} where the sum is over multiindices satisfying
$$0<i_1<j_1<..<j_r,\quad\sum (j_{*}-i_{*})=n.$$
If $n_k=j_k-i_k$, then $\sum n_k=n$, but this sequence is not
necessarily decreasing. However if we eliminate the constraints
$i_1<..<i_r$, then we can permute the indices ($1\le k\le r$) for
the $i_k$ and $n_k$. We can estimate that
(\ref{coefficient}) is
$$\le \sum_{(n_i)\in \mathcal P (n)} \sum_{i_1,..,i_l>0}
\vert\zeta_{i_1}\vert
\vert\zeta_{i_1+n_1}\vert..\vert\zeta_{i_l}\vert
\vert\zeta_{i_l+n_l}\vert= \sum_{(n_i)\in \mathcal P (n)}
\prod_{s=1}^l \sum_{i_s>0}\vert\zeta_{i_s}\vert
\vert\zeta_{i_s+n_s}\vert$$
$$\le\sum_{\mathcal P(n)}\vert \zeta \vert_{l^2}^{2l((n_i))}$$
This shows that the Taylor coefficients of any limiting function
for the $\delta)2^{(N)}$ are unique. The same considerations apply to the
$\gamma_2^{(N)}$. Thus the sequences $(\gamma_2^{(N)})$ and
$(\delta_2^{(N)})$ converge uniformly on compact sets of $\Delta$ to
unique limiting functions. This proves our claim about uniqueness
of the limits $c_2$ and $d_2$.

The fact that $k_2$ actually has
values in $SU(2)$ is a consequence of Lemma \ref{subseqlemma}.
This completes the proof of Lemma \ref{prequellemma}.
\end{proof}

We have now proven the existence of a
$$k_2(\zeta)=\left(\begin{matrix}d_2^*&-c_2^*\\c_2&d_2\end{matrix}\right)$$
as in (II.1), but we have not proven the simultaneous nonvanishing of $c_2$ and $d_2$.

We now want to investigate the existence of a triangular factorization
$$k_2=\left(\begin{matrix} 1&x^*\\
0&1\end{matrix} \right)\left(\begin{matrix}\mathbf a_2&0\\
0&\mathbf a_2^{-1}\end{matrix} \right)\left(\begin{matrix} \alpha_2 (z)&\beta_2 (z)\\
\gamma_2 (z)&\delta_2 (z)\end{matrix} \right)$$
where $\mathbf a_2>0$. Note we have explicit formulas for $\mathbf a_2,\gamma_2$ and $\delta_2$. But we need a formula for $x$. If we can find $x$, then we can use (\ref{4theqn}) to find $\alpha_2,\beta_2$. Because
of the identity (\ref{normidentity}) it would only remain to show $x$ is square integrable.

Recall from the appendix to \cite{Pi1} that $x^{*}$ has the form
$$x^{*}=\sum_{j=1}^{\infty}x_1^{*}(\zeta_j,..)z^{-j},$$
where
$$x_1^{*}(\zeta_1,..)=\sum_{n=1}^{\infty}\zeta_n\left(\prod_{k=n+1}^{\infty}
(1+\vert\zeta_k\vert^2)\right)s_n(\zeta_n,\zeta_{n+1},\bar{\zeta}_{n+1},
..),$$ $s_1=1$ and for $n>1$,
$$s_n=\sum_{r=1}^{n-1}s_{n,r},\quad s_{n,r}=\sum c_{i,j}\zeta_{i_1}\bar{\zeta}_{j_1}\zeta_{
i_2}\bar{\zeta}_{j_2}..\zeta_{i_r}\bar{\zeta}_{j_r}$$ where the
sum is over multiindices satisfying the constraints
\begin{equation}\label{index}\begin{matrix} &&j_1&\le&..&\le j_r&\\
&&\lor&&&\lor\\
n&\le&i_1&\le&..&i_r&\end{matrix}
,\quad\sum_{l=1}^r(j_l-i_l)=n-1,\end{equation}
The crucial point is that the $c_{i,j}$ are positive integers, although it is not known
how to explicitly compute them. In particular for each $n$ $s_n$ contains the subsum
$ \sum_{m\ge n}\zeta_m\zeta_{m+n-1}^*$.

Now suppose that all of the $\zeta_n\ge 0$. If the sum for $x_1^*$ converges, then the sum
$$\sum_{n=1}^{\infty}\zeta_n\sum_{m\ge n}\zeta_m\zeta_{m+n-1}^* $$
has to converge. But $\zeta\in l^2$ is not a sufficient condition to guarantee the convergence of
this sum. Empirically, if $\zeta_n=n^{-p}$ with $p<5/8$, the sum diverges. From a theoretical point of view,
this is the convolution of three functions on $\mathbb Z$ evaluated at zero, $\zeta^{t}*\zeta^{t}*\zeta$,
where $\zeta^{t}(-m)=\zeta(m)$ is the adjoint; the convolution of two $l^2(\mathbb Z)$ functions only has the property that
it vanishes at infinity, and the convolution of an $l^2(\mathbb Z)$ function and a function that vanishes at infinity
is not generally defined.
This explains why (II.2) in Question \ref{introtheorem3} does not imply (II.3).

This gap can possibly be (partially) filled by the following hybrid deterministic/probabilistic

\begin{conjecture}\label{intro3} In reference to Question \ref{introtheorem3}, if $\zeta\in l^2$ as in (II.2)
and the phases of the $\zeta_k$ are uniform and independent as random variables, then $k_2$ has a triangular
factorization as in (II.3).
\end{conjecture}

To get started on this, we would need to prove the almost sure existence of $x_1$ above. This has not been
done. Instead we will explain the meaning of the operators in the statement of Question \ref{introtheorem3},
which should play an important role in the proof of the conjecture.

\begin{lemma}\label{inversever1} For sufficiently regular $x$ (which we will clarify in the proof)
$$\mathbf a_2^2=\frac{det(1+\dot B(x)\dot B(x)^*)}{det(1+\dot B(z^{-1}x)\dot B(z^{-1}x)^*)} $$
$$=1+\langle x|(1+\dot B(z^{-1}x)\dot B(z^{-1}x)^*)^{-1}x\rangle_{L^2}
=\frac{1}{\langle 1|(1+\dot B(x)\dot B(x)^*)^{-1}1\rangle_{L^2}}$$
($\langle\cdot\cdot\rangle$ is the $L^2$ inner product), where $\dot B(x)$ denotes the scalar Hankel operator
corresponding to the symbol $x$.
\end{lemma}

\begin{proof} For the first equality see (2.13) of \cite{Pi1}. For the determinants in this formula to make sense,
we need $\zeta\in\mathbf w^{1/2}$.

As a matrix (relative to the standard Fourier basis)
$$\dot B(x)\dot B(x)^*-\dot B(z^{-1}x)\dot B(z^{-1}x)^*=(x_nx_m^*)_{n,m\ge 1} $$
because the $n,m$ entry is
$$\sum_{i\ge 0}(x_{n+i}x_{m+i}^*-\sum_{i\ge 0}(x_{n+1+i}x_{m+1+i}^*=x_nx_m^*$$
This is a rank one matrix.

The identity
$$(1+S)(1+T)^{-1}=1+(T-S)(1+T)^{-1}$$
implies
$$(1+\dot B(x)\dot B(x)^*)(1+\dot B(z^{-1}x)\dot B(z^{-1}x)^*)^{-1}$$
$$=1+\left(\dot B(x)\dot B(x)^*-\dot B(z^{-1}x)\dot B(z^{-1}x)^*\right) (1+\dot B(z^{-1}x)\dot B(z^{-1}x)^*)^{-1}$$
This is a rank one perturbation of the identity, and the determinant equals
$$1+\langle x|(1+\dot B(z^{-1}x)\dot B(z^{-1}x)^*)^{-1}x\rangle_{L^2}$$
This proves the second equality. This second formula has a transparent operator-theoretic meaning when the Hankel operator
is bounded, and this is the case if $x\in BMO$.

For the third equality, suppose that $x=\sum_{n\ge 1}x_nz^n\in L^2$. For $i,j\ge 0$, relative to
the standard Fourier basis $z^0,z^1,...$ for $\dot H^+$, the $i,j$ entry for the matrix
representing $\dot B(x)\dot B(x)^*$ equals
\begin{equation}\label{matrixformula}\sum_{n=0}^{\infty}x_{i+n}x_{j+n}^* \end{equation}
The matrix representing $\dot B(z^{-1}x)\dot B(z^{-1}x)^*$ (aside from indexing) is the same
as the matrix obtained by deleting the zeroth row and column of the matrix representing
$\dot B(x)\dot B(x)^*$. Thus the third equality is simply Cramer's rule for the inverse.
The use of this rule is valid provided $\zeta\in\mathbf w^{1/2}$, which guarantees the determinants make sense.
However as a formula for $\mathbf a_2$, it has a transparent operator-theoretic meaning when $x\in BMO$.
In the next lemma we will see the formula
makes sense for $(x_n)\in l^2$.

\end{proof}

We will now sharpen this result.

\begin{lemma}\label{inverse} Suppose that $\zeta\in l^2$.

(a) The sequence of positive operators $(1+\dot B(x^{(n)})\dot B(x^{(n)})^*)^{-1}$ has a unique
norm operator limit, and it is given by the formula
$$(1+\dot B\dot B^*)^{-1}f=c_2(c_2^*f)_{0+}+d_2(d_2^*f)_{0+} $$
[$x$ does not appear in the notation, to emphasize that we are not assuming the
existence of $x$]. Also
$$ A(k_2)A(k_2*)\left(\begin{matrix}f_1\\f_2\end{matrix}\right)
=\left(\begin{matrix}f_1\\(1+\dot B\dot B^*)^{-1}f_2\end{matrix}\right) $$

Similarly the sequence of positive operators $(1+\dot B(x^{(n)})^*\dot B(x^{(n)}))^{-1}$ has a
norm operator limit. This limit is unique and denoted by $(1+\dot B^*\dot B)^{-1}$.

(b)
$$(1+\dot B\dot B^*)^{-1}(z^n)=\mathbf a_2^{-2}(\gamma_2\sum_{j=0}^{n-1}\gamma_{2,n-j}^*z^j+\delta_2\sum_{k=0}^n\delta_{2,n-j}^*z^j$$
$$=\mathbf a_2^{-2}(z^n\gamma_{2}^{(n-1)*}\gamma_2+z^n\delta_{2}^{(n-1)*}\delta_2 $$
For example
$$(1+\dot B\dot B^*)^{-1}(1)=\mathbf a_2^{-2}\delta_2$$
$$(1+\dot B\dot B^*)^{-1}(z)=\mathbf a_2^{-2}(\gamma_{2,1}^*\gamma_2+(\delta_{2,1}^*+z)\delta_2)$$
$$(1+\dot B\dot B^*)^{-1}(z^2)=\mathbf a_2^{-2}((\gamma_{2,2}^*+\gamma_{2,1}^*z)\gamma_2+(\delta_{2,2}^*+\delta_{2,1}^*z+z^2)\delta_2)$$
and the diagonal is
$$\mathbf a_2^{-2}diag(1,1+|\gamma_{2,1}|^2+|\delta_{2,1}|^2,...,1+\sum_{k=1}^n(|\gamma_{2,k}|^2+|\delta_{2,k}|^2),...) $$

(c) If $x$ is $l^2$ and $n\ge -1$, then
$$(1+\dot B\dot B^*)^{-1}(z^nx)=-\gamma_2z^n\alpha_2^{(n)*}-\delta_2z^n\beta_2^{(n)*} $$
in particular
$$(1+\dot B\dot B^*)^{-1}(z^{-1}x)=-z^{-1}\gamma_2 $$
or equivalently
$$(1+\dot B^*\dot B)^{-1}x^*=-\mathbf a_2^{-2}\gamma_2^*$$

\end{lemma}

\begin{remark}Parts (b) and (c) explicitly determine $\gamma_2$ and $\delta_2$ in terms of $x$. This
explains the meaning of the formulas in Theorem \ref{keyidentities1}. \end{remark}

\begin{proof}
(a) Since $1+\dot B(x^{(n)})^*\dot B(x^{(n)})\ge 1$, it follows that the sequence $(1+\dot B(x^{(n)})^*\dot B(x^{(n)}))^{-1}$ has strong operator limits.  We must prove uniqueness. For this it will suffice to
prove the exact formula for $x\in L^2$, because using this formula we can take a limit to obtain the general formula.
After discussing the calculations in (c) and (d), we will then explain why this is actually a norm operator limit.

We need several standard facts: (1) If $g=g_-g_0g_+$, then $Z(g):=C(g)A(g)^{-1}=Z(g_-)$. (2) If $g$ is unitary,
then $(1+Z^*Z)^{-1}=A(g)A(g^{-1})$. And (3) If $g_-=\left(\begin{matrix}1&x^*\\0&1\end{matrix}\right)$, then
$$Z(g_-)\left(\begin{matrix}f_1\\f_2\end{matrix}\right)=\left(\begin{matrix}C(x^*)f_2\\0\end{matrix}\right)$$
It is straightforward to check (1). (2) follows from (1). And (3) is straightforward.

Now suppose that $g=k_2$ and $k_2$ has a triangular factorization. By (2)
$$(1+Z^*Z)^{-1}\left(\begin{matrix}f_1\\f_2\end{matrix}\right)=A(k_2)A(k_2*)\left(\begin{matrix}f_1\\f_2\end{matrix}\right)
=\left(\begin{matrix}f_1\\f_2-(c_2(c_2^*f_2)_{-})_{0+}-(d_2(d_2^*f_2)_{-})_{0+}\end{matrix}\right) $$
Now (3) implies
$$(1+\dot B\dot B^*)^{-1}=1-(B(c_2)B(c_2)^*+B(d_2)B(d_2)^*=A(c_2)A(c_2)^*+ A(d_2)A(d_2)^*$$
This formula does not depend on the assumption that $k_2$ has a triangular factorization (hence we can apply the formula to $k_2^{(n)}$ and take a limit). This formula is equivalent to the one in the statement of part (a) of the theorem.

The calculations in (b) are straightforward, given the formula in (a). The calculations in (c) also use
the unitarity equation $\mathbf a_2^2\alpha_2^*+\gamma_2^*=\delta_2$, multiplied by $z^n$. Together with the
formula in (a) this implies
$$(1+\dot B\dot B^*)^{-1}(z^nx)=\mathbf a_2^{-2}(\gamma_2(z^n\delta_2-\mathbf a_2^2z^n\alpha_2^{(n)*})-\delta_2(z^n\gamma_2+\mathbf a_2^2z^n\beta_2^{(n)*}) $$
This simplifies to the formula in (c).

Finally we explain why the limits in (a) are actually norm limits. Note that $(1+\dot B\dot B^*)^{-1}\le 1$
as positive operators. The formula for the diagonal in part (b) shows that the diagonal entries monotonely
increase to $1$ as $n\to \infty$. This implies uniform convergence.

\end{proof}

\begin{question} If $\zeta\in l^2$, then $0\le (1+\dot B\dot B^*)^{-1}\le 1$. Is $(1+\dot B\dot B^*)^{-1}$ injective? What can we say about the spectrum of  $(1+\dot B\dot B^*)^{-1}$? If $x\in VMO$, then the spectrum is discrete. Does the spectrum simply become continuous on $[0,1]$ outside of VMO?

\end{question}

Here is a naive $L^2$ analogue of Theorem \ref{introtheorem2}.

\begin{question}\label{introtheorem4}  Suppose that $g:S^1 \to SU(2)$ is measurable. Are the following conditions equivalent:

(i) $A(g)$ and $A_1(g)$ are invertible.

(ii) $g$ has a triangular factorization.

(iii) $g$ {\bf and} $g^{-1}$ have (root subgroup) factorizations of the form
$$g= k_1(\eta)^* \left(\begin{matrix}e^{\chi}&0\\0&e^{-\chi}\end{matrix}\right)k_2(\zeta)$$
$$g^{-1}= k_1(\eta')^* \left(\begin{matrix}e^{\chi'}&0\\0&e^{-\chi'}\end{matrix}\right)k_2(\zeta')$$
where $k_1$ and $k_2$ are as in (some form of) Question \ref{introtheorem3}, and
$exp(-\chi_+), exp(-\chi_+')\in L^2$.

\end{question}

\begin{remark} The conditions (i) and (ii) are invariant with respect
interchange of $g$ and $g^{-1}$ (This depends on $A(g^{-1})=A(g^*)=A(g)^*$ (and similarly for $A_1$),
and $g^{-1}=u(g)^*m(g)^*a(g)l(g)^*$. It is for this reason that we have imposed a condition on
both $g$ and its inverse in part (iii). This was not necessary in
the $W^{1/2}$ case.

\end{remark}

In the remainder of the section, we will explain how these statements have to be modified.

First, it is known that (i) is equivalent to \\

(ii') $g$ has a triangular factorization, $g=lmau$, {\bf and} the operators
$$R:\mathbb C[z]\otimes \mathbb C^2 \to \mathbb C[z]\otimes \mathbb C^2:\psi_+ \to M_{u^{-1}}\circ P_{+}\circ M_{l^{-1}}(\psi_+)$$
(where $P_+$ is either the projection for the polarization (\ref{polarization}) or the shifted polarization)
extend to bounded operators. \\

This is a special case of Theorem 5.1 (page 109) of \cite{Peller}, which establishes a criterion for invertibility of $A(g)$ for more general essentially bounded matrix symbols.

\begin{theorem} If $k_1,k_1',k_2$ and $k_2'$ have triangular factorizations (as in (I.3) and (II.3) of Question \ref{introtheorem3}, then $g$ has a triangular factorization (as in (ii) of Question \ref{introtheorem4})
\end{theorem}

\begin{proof} We will recall some more formulas which relate triangular and root subgroup factorization.

\begin{proposition}\label{trifactorization} Suppose that $\eta,\chi,\zeta$ are sufficiently regular (e.g. $\mathbf w^{1/2}$) Then $g=k_1^*e^{\chi}k_2$ has triangular factorization $g=l(g)m(g)a(g)u(g)$, where
$$l(g)=\left(\begin{matrix}l_{11}&l_{12}\\
l_{21}&l_{22}\end{matrix}\right)=\left(\begin{matrix} \alpha_1^{*}&-(Y^*\alpha_1)_{-}\\
\beta_1^{*}&1-(Y^*\beta_1)_-\end{matrix}
\right)\left(\begin{matrix}e^{-\chi_+^*}&0
\\0&e^{\chi_+^*}\end{matrix}\right)
\left(\begin{matrix}1&M_-\\0&1\end{matrix}\right)
$$

$$m(g)=\left(\begin{matrix}e^{\chi_0}&0\\0& e^{-\chi_0}\end{matrix}\right),
\quad a(g)=\left(\begin{matrix}a_0&0\\0&a_0^{-1}\end{matrix}\right)=
\left(\begin{matrix}a_1a_2&0\\0&(a_1a_2)^{-1}\end{matrix}\right)$$

$$u(g)=\left(\begin{matrix}u_{11}&u_{12}\\
u_{21}&u_{22}\end{matrix}\right)=\left(\begin{matrix} 1&M_{0+}\\
0&1\end{matrix} \right)\left(\begin{matrix}e^{\chi_+}&0
\\0&e^{-\chi_+}\end{matrix}\right)\left(\begin{matrix} 1-(X^*\gamma_2)_{+}&-(X^*\delta_2)_{0+}\\
\gamma_2&\delta_2\end{matrix}\right) $$
$Y=\mathbf a_1^{2}y$, $X=\mathbf a_2^{-2}x$, and
$$M=(a_0m_0)^{-2}e^{2\chi_+^*}Y+e^{2\chi_+}X^{*}$$

\end{proposition}

\begin{proof} Given the triangular factorizations for $k_1$
and $k_2$, $g$ equals
$$\left(\begin{matrix} \alpha_1&\beta_1\\
\gamma_1&\delta_1\end{matrix}
\right)^*\left(\begin{matrix}1&Y\\0&1\end{matrix}\right)\left(\begin{matrix}a_1a_2e^{-\chi_+^*+\chi_0+\chi_+}&0
\\0&(a_1a_2e^{-\chi_+^*+\chi_0+\chi_+})^{-1}\end{matrix}\right)
\left(\begin{matrix} 1&X^{*}\\
0&1\end{matrix} \right)\left(\begin{matrix} \alpha_2&\beta_2\\
\gamma_2&\delta_2\end{matrix} \right)$$

\begin{equation}\label{triformula}=\left(\begin{matrix} \alpha_1^*&\gamma_1^*\\
\beta_1^*&\delta_1^*\end{matrix}
\right)\left(\begin{matrix}e^{-\chi_+^*}&0
\\0&e^{\chi_+^*}\end{matrix}\right)\end{equation}
$$\left(\begin{matrix}1&e^{2\chi_+^*}Y\\0&1\end{matrix}\right)
\left(\begin{matrix}a_1a_2e^{\chi_0}&0
\\0&(a_1a_2e^{\chi_0})^{-1}\end{matrix}\right)
\left(\begin{matrix} 1&e^{2\chi_+}X^{*}\\
0&1\end{matrix} \right)\left(\begin{matrix}e^{\chi_+}&0
\\0&e^{-\chi_+}\end{matrix}\right)\left(\begin{matrix} \alpha_2&\beta_2\\
\gamma_2&\delta_2\end{matrix} \right)$$ The product of the middle
three factors is upper triangular, and it is easy to find its
triangular factorization:
$$\left(\begin{matrix}1&e^{2\chi_+^*}Y\\0&1\end{matrix}\right)
\left(\begin{matrix}a_1a_2e^{\chi_0}&0
\\0&(a_1a_2e^{\chi_0})^{-1}\end{matrix}\right)
\left(\begin{matrix} 1&e^{2\chi_+}X^{*}\\
0&1\end{matrix} \right)$$
$$=\left(\begin{matrix}1&M_-\\0&1\end{matrix}\right)
\left(\begin{matrix}a_1a_2e^{\chi_0}&0\\0& (a_1a_2e^{\chi_0})^{-1}\end{matrix}\right)\left(\begin{matrix} 1&M_{0+}\\
0&1\end{matrix} \right) $$
where
$$M=(a_1a_2)^{-2}e^{-2(-\chi_+^*+\chi_0)}Y+e^{2\chi_+}X^{*} $$

It remains to explain the formulas in the proposition for $\gamma_1^*$, $\delta_1^*$, $\alpha_2$ and
$\beta_2$ in (\ref{triformula}). The unitarity equations for $k_2$
imply
$$\alpha_2=-a_2^{-2}x^*\gamma_2+a_2^{-2}\delta_2^*\quad
\text{and}\quad \beta_2=-a_2^{-2}x^*\delta_2-a_2^{-2}\gamma_2^*$$
Applying the $(\cdot)_{0+}$ projection to each of these,  we obtain
$\alpha_2=1-(X^*\gamma_2)_+$ and $\beta_2=-(X^*\delta_2)_{0+}$.

The formulas for $\gamma_1$ and $\delta_1$ are derived in a similar way.
\end{proof}

We claim that the formulas in the Proposition yield
a triangular factorization for $g$.
We need to show that the $l(g)$ and $u(g)$ factors are $L^2$. On $S^1$
$$|\alpha_1|^2+|\beta_1|^2=a_1^{-2} \text{ and }|\gamma_2|^2+|\delta_2|^2=a_2^2 $$
Consequently the first column of $l(g)$ and the second row of $u(g)$ are $L^2$ iff
$$exp(Re(\chi_-))=exp(-Re(\chi_+)) \in L^2 $$
We are assuming this in (iii), and hence the first column of $l(g)$ and the
second row of $u(g)$ are $L^2$.

The second column of $l(g)$ and the first row of $u(g)$ appear to be hopeless. But here is the key fact:
$g$ has a triangular factorization iff $g^{-1}$ has a triangular factorization (If $g=lmau$, then $g^{-1}=u(g)^*m(g)^*a(g)l(g)^*$). Moreover the problematic second column for $l(g)$ is the adjoint of the second row of $u(g^{-1})$,
and similarly the problematic first row of $u(g)$ is the adjoint of the first column of $l(g^{-1})$. It is not a priori clear (and it is undoubtedly not true) that for a general measurable $g:S^1 \to SU(2)$, $g$ has a root subgroup factorization iff $g^{-1}$ has a root subgroup factorization. But we do not have a concrete example to offer. In (iii), we are assuming both $g$ and $g^{-1}$ have root subgroup factorizations. Consequently the second column of $l(g)$ and the
first row of $u(g)$ are also $L^2$ Thus $g$ has a triangular factorization as in (ii).
\end{proof}

\begin{theorem}\label{L2theorem}Assume that $g$ has a triangular factorization. Then $g$ (and $g^{-1}$) have root subgroup
factorizations as in (iii), where in (iii) we mean in the sense of (I.1) and (II.1) of Question \ref{introtheorem3}.

\end{theorem}

\begin{proof} Although somewhat longwinded, it is straightforward to use the formulas
in Proposition \ref{trifactorization} to find candidates for the
factors $k_1,\chi$ and $k_2$, see (3.4)-(3.19)
of \cite{Pi1} (when consulting these formulas, note that the $\chi_+$ of this paper is denoted by
$\chi$ in \cite{Pi1}). We will now list these formulas, explain why they make sense, and note their
significance:
\begin{equation}\mathbf a_1=exp(-\frac1{4\pi}\int_{S^1}log(\vert l_{11} \vert
^2+ \vert l_{21} \vert ^2)d\theta),\end{equation}
\begin{equation}
\mathbf a_2=exp(\frac1{4\pi}\int_{S^1}log(\vert u_{21} \vert ^2+\vert
u_{22} \vert ^2)d\theta),\end{equation}
We claim these are finite positive numbers. By assumption $l,u$ are square integrable around $S^1$,
and $0<\mathbf a=\mathbf a_1\mathbf a_2<\infty$. This implies $\mathbf a_1$ and $\mathbf a_2$ are nonzero.
Jensen's inequality implies
$$ \mathbf a_2^2\le \int_{S^1}(\vert u_{21} \vert ^2+\vert
u_{22} \vert ^2)\frac{d\theta}{2\pi}<\infty$$
Thus $\mathbf a_1$ and $\mathbf a_2$ are finite. This proves the claim.

On $S^1$,
\begin{equation}\vert l_{11} \vert ^2+ \vert l_{21} \vert ^2= \mathbf a_1^{-2} exp(2Re(\chi_-))\end{equation}
\begin{equation}\label{chieqn}\vert u_{21} \vert ^2+\vert u_{22} \vert^2=\mathbf a_2^2 exp(-2Re(\chi_+))\end{equation}
These formulas imply $\exp(-\chi_+)\in L^2$, as in (iii).
\begin{equation}\label{1}l_{11}=\alpha_1^*exp(\chi_-),\quad l_{21}=\beta_1^*exp(\chi_-)\end{equation}
and
\begin{equation}\label{2}u_{21}=\gamma_2exp(-\chi_+),\quad u_{22}=\delta_2exp(-\chi_+)\end{equation}
and on $S^1$,
\begin{equation}\vert \alpha_1 \vert ^2+\vert \beta_1 \vert ^2=\mathbf a_1^{-2}\end{equation}
\begin{equation}\vert \delta_2 \vert ^2+\vert \gamma_2 \vert ^2=\mathbf a_2^{2}\end{equation}
These formulas enable us to recover measurable loops $k_1,k_2:S^1 \to SU(2)$,
$$k_1=\mathbf a_1\left(\begin{matrix}\alpha_1&\beta_1\\-\beta_1^*&\alpha_1^*\end{matrix}\right) \text{ and }
k_2=\mathbf a_2^{-1}\left(\begin{matrix}\delta_2^*&-\gamma_2^*\\\gamma_2&\delta_2\end{matrix}\right) $$
Because $l^*$ is invertible at all points of
$\Delta$, (\ref{1}) implies that the entries $a_1$ and $b_1$ of $k_1$ do not simultaneously
vanish. Similarly, because $u$ is invertible, (\ref{2}) implies that the entries $c_2$ and
$d_2$ do not simultaneously vanish. Using Theorem \ref{introtheorem3} we can obtain $\eta$ and $\zeta$ from
the Taylor series expansions of $\beta_2/\alpha_2$ and $\gamma_2/\delta_2$, and $\eta$ and $\zeta$
are in $l^2$ because of the finiteness of $\mathbf a_1$ and $\mathbf a_2$.
\end{proof}

One of several shortcomings of this theorem is that we have assumed that both $g$ and $g^{-1}$ have root subgroup factorizations. This is undesirable because there are (hopelessly) complicated compatibility relations involving
the pairs of parameters $\eta, \zeta,\chi$ and $\eta',\chi',\zeta'$, for $g$ and $g^{-1}$, respectively.

\begin{example} Suppose that $g=k_2(\zeta)$, i.e. $\eta$ and $\chi$ are zero. In this case
$g^{-1}$ has the triangular decomposition
$$g^{-1}=g^*=\left(\begin{matrix} \alpha_2^* (z)&\gamma_2^* (z)\\
\beta_2^* (z)&\delta_2^* (z)\end{matrix} \right)\left(\begin{matrix}\mathbf a_2&0\\
0&\mathbf a_2^{-1}\end{matrix} \right)\left(\begin{matrix} 1&0\\
x&1\end{matrix} \right)
$$
Therefore
$$\frac{\gamma_2(g^{-1})}{\delta_2(g^{-1})}(z)=x_1(\zeta_1,...)z^1+...$$
implying $\xi_n(g^{-1})=x_n(g)$ and in particular
$$-\zeta_1(g^*)=x_1^*(\zeta_1,...) $$
The formula for $x_1^*$ is discussed in an appendix - suffice it to say, it is complicated.

\end{example}

\section{The $VMO$ Theory}\label{VMO}

In this section we will consider VMO loops and compact operators. Everything we say can
be generalized to Besov class $B^{1/p}_p$ loops and Schatten $p$-class operators. For simplicity of exposition
we will focus on the maximal class, VMO.

We begin by recalling basic facts about the abelian case, $VMO(S^1,S^1)$.
The notion of degree (or winding number) can be extended from $C^0$ to
$VMO(S^1,S^1)$ (see Section 3 of \cite{Brezis} for an amazing
variety of formulas, and further references, or pages 98-100 of
\cite{Peller}). Also given $\lambda \in VMO(S^1,S^1)$, we view
$\lambda$ as a multiplication operator on $H=L^2(S^1)$, with the
Hardy polarization. We write $\dot A(\lambda)$ for the Toeplitz
operator, and so on (with the dot), to avoid confusion with the
matrix case.

\begin{lemma}\label{abelian} There is an exact sequence of topological groups
$$0 \to 2\pi i\mathbb Z \to VMO(S^1,i\mathbb R)
\stackrel{exp}{\rightarrow}
VMO(S^1,S^1)\stackrel{degree}{\rightarrow} \mathbb Z \to 0.$$
Moreover $degree(\lambda)=-index(\dot A(\lambda))$.
\end{lemma}

This is implicit on pages 100-101 of \cite{Peller}. The important point is that a $VMO$ function cannot have jump discontinuities. This implies that the kernel of $exp$ is $2 \pi i \mathbb Z$. Thus the
sequence in the statement of the Lemma is continuous and exact.

\begin{remark} This should be contrasted with the measurable case. The short exact sequence
$0\to \mathbb Z\to \mathbb R \to \mathbb T\to 0$ induces a short exact sequence of Polish topological groups
$$0 \to Meas([0,1],\mathbb Z) \to Meas([0,1],\mathbb R) \to Meas([0,1],\mathbb T) \to 0$$ (see Section 2, especially Proposition 9, of \cite{Moore}). However $Meas([0,1],\mathbb Z)$ is not discrete, and (just as the unitary group of an infinite dimensional Hilbert space is contractible - in either the strong operator or norm topology) $Meas([0,1],\mathbb T)$ is contractible.
\end{remark}

Our aim now is to specialize Theorems \ref{introtheorem1} and \ref{introtheorem2} to VMO loops. It seems
unlikely that one can characterize the sequences $\eta$ and $\zeta$ that will correspond to VMO
loops $k_1,k_2:S^1 \to SU(2)$, respectively, as in Theorem \ref{introtheorem1} (so far as we are aware,
it is not known how to characterize scalar VMO functions in terms of their Fourier coefficients).
For this reason we will use $y\in VMO_{0+}$ and $x\in VMO_+$ as parameters.

\begin{proposition} Suppose $k_2:S^1\to SU(2)$. The following two conditions are equivalent:

(II.1) $k_2\in VMO$ is of the form
$$k_2(z)=\left(\begin{matrix} d^{*}(z)&-c^{*}(z)\\
c(z)&d(z)\end{matrix} \right),\quad z\in S^1,$$ where $c,d\in
H^0(\Delta)$, $c(0)=0$, $d(0)>0$, $c$ and $d$ do not
simultaneously vanish at a point in $\Delta$.

(II.3) $k_2$ has triangular factorization of the form
$$\left(\begin{matrix} 1&\sum_{j=1}^{\infty}x^*_jz^{-j}\\
0&1\end{matrix} \right)\left(\begin{matrix}\mathbf a_2&0\\
0&\mathbf a_2^{-1}\end{matrix} \right)\left(\begin{matrix} \alpha_2 (z)&\beta_2 (z)\\
\gamma_2 (z)&\delta_2 (z)\end{matrix} \right)$$
where $\mathbf a_2>0$ and $\gamma_2,\delta_2\in VMO$.

There is a similar equivalence for $k_1$.

\end{proposition}

\begin{proof}
The equivalence of II.1 and II.3 is proven exactly as in the $W^{1/2}$ case (taking into account
the VMO condition of $\gamma_2$ and $\delta_2$ in II.3). This uses the invertibility of $A(k_2)$. For the injectivity of $A(k_2)$, see Lemma \ref{injectivelemma} below (which is more general). The surjectivity follows since
the VMO condition implies that $A(k_2)$ is Fredholm of index zero.
\end{proof}

\begin{theorem}\label{introtheorem5} (a) For $k_2$ in the preceding proposition, $x\in VMO_+$.

(b)  The map $k_2 \to x$ induces a bijection
$$\{k_2: \text{ II.1 and II.3 hold }\} \leftrightarrow VMO_+: k_2 \leftrightarrow x$$

(c) In terms of the root subgroup factorization in Theorem \ref{rootsubgpfact}, the singular inner function $\lambda=1$.

There is a similar statement for $k_1$.
 \end{theorem}

\begin{proof} (a) The operator $(1+\dot B\dot B^*)^{-1}$ is essentially the product $A(k_2)A(k_2^{-1})$, which is of the form $1+$compact operator. Thus the inverse $1+\dot B(x)\dot B(x)^*$ is also a compact perturbation of the identity. This is equivalent to $x\in VMO_+$. This proves part (a).

To prove part (b), we simply run the argument the opposite direction: if $x\in VMO_+$, then $(1+\dot B\dot B^*)^{-1}$,
hence also $A(k_2)A(k_2^{-1})$,
is a compact perturbation of the identity. This implies $k_2$ is VMO.

(c) For the Caratheodory function in (\ref{Caratheodory}) to be VMO, $\nu$ has to be absolutely continuous with respect to Lebesgue measure. Hence $\lambda=1$.

\end{proof}

\begin{question} In reference to (a), it is interesting to ask if $\alpha_2$ and $\beta_2$ are VMO.
The analogous question is also open in the $W^{1/2}$ case.
\end{question}

There is a gap in our proof of the following, see Lemma \ref{deformation}.

\begin{theorem}\label{introtheorem6} Suppose that $g\in VMO(S^1,SU(2))$ and assume the truth of Lemma
\ref{deformation} below. The following are equivalent:

(a) $A(g)$ and $A_1(g)$ are invertible.

(b) $g$ has a triangular factorization.

(c) $g$  has a (root subgroup) factorization of the form
$$g= k_1(\eta)^* \left(\begin{matrix}e^{\chi}&0\\0&e^{-\chi}\end{matrix}\right)k_2(\zeta)$$
where $k_1$ and $k_2$ are as in Theorem \ref{introtheorem5}, $\chi\in VMO(S^1;i\mathbb R)$
and $exp(- \chi_+)\in L^2(S^1)$.

\end{theorem}

\begin{proof} The equivalence of (a) and (b) is true more generally for $g\in QC(S^1,SL(2,\mathbb C))$ (see (b) of Remark of \cite{Pi1}). \

To see that (a) and (b) are equivalent to (c), we will need some lemmas. To simplify the notation,
let $h_1=\left(\begin{matrix}1&0\\0&1\end{matrix}\right)$.

\begin{lemma}\label{Toeplitzfactorization} With appropriate domains
 $$ A(k_1^* e^{\chi h_1}k_2)=A(k_1^* e^{\chi_-h_1})A(e^{\chi_{0+}h_1}k_2)$$
The same is true for $A_1$ in place of $A$.
Similarly
$$ D(k_1^* e^{\chi h_1}k_2)=D(k_1^* e^{\chi_-h_1})D(e^{\chi_{0+}h_1}k_2)$$
and the same is true for $D_1$ in place of $D$.
\end{lemma}

\begin{proof}  The first statement is equivalent to showing that
$B(k_1^* e^{\chi_-h_1})C(e^{\chi_{0+}h_1}k_2)$ vanishes. Applied to $\left(\begin{matrix}f_1\\f_2\end{matrix}\right)\in H_+$, this equals
$$B(\left(\begin{matrix}e^{\chi_-}a_1^*&-e^{-\chi_-}b_1\\e^{\chi_-}b_1^*&e^{-\chi_-}a_1\end{matrix}\right))
C(\left(\begin{matrix}e^{\chi_+}d_2^*&-e^{-\chi_+}c_2^*\\e^{-\chi_+}c_2&e^{-\chi_+}d_2\end{matrix}\right))
\left(\begin{matrix}f_1\\f_2\end{matrix}\right)$$
$$=\left[ \left(\begin{matrix}e^{\chi_-}a_1^*&-e^{-\chi_-}b_1\\e^{\chi_-}b_1^*&e^{-\chi_-}a_1\end{matrix}\right)
\left(\begin{matrix}(e^{\chi_+}d_2^*f_1-e^{-\chi_+}c_2^*f_2)_-\\0\end{matrix}\right)\right]_+$$
$$=\left[\left(\begin{matrix}e^{\chi_-}a_1^*(e^{\chi_+}d_2^*f_1-e^{-\chi_+}c_2^*f_2)_-\\
e^{\chi_-}b_1^*(e^{\chi_+}d_2^*f_1-e^{-\chi_+}c_2^*f_2)_-\end{matrix}\right)\right]_+=0 $$
This proves the first statement.

For the second statement involving $A_1$, we are considering a polarization
for $H$ where $H_+$ now has orthonormal basis $\{\epsilon_iz^j:i=1,2,j>0\}\cup\{\epsilon_1\}$ (see (\ref{basis})). We let
$B_1,C_1$ denote the Hankel operators relative to this shifted polarization. We must show
$B_1(k_1^* e^{\chi_-h_1})C_1(e^{\chi_{0+}h_1}k_2)$ vanishes. The calculation is basically the same, but it depends on
our normalizations in a subtle way. Applied to $\left(\begin{matrix}f_1\\f_2\end{matrix}\right)\in H_+$, this equals
$$B_1(\left(\begin{matrix}e^{\chi_-}a_1^*&-e^{-\chi_-}b_1\\e^{\chi_-}b_1^*&e^{-\chi_-}a_1\end{matrix}\right))
C_1(\left(\begin{matrix}e^{\chi_+}d_2^*&-e^{-\chi_+}c_2^*\\e^{-\chi_+}c_2&e^{-\chi_+}d_2\end{matrix}\right))
\left(\begin{matrix}f_1\\f_2\end{matrix}\right)$$
$$=B_1(\left(\begin{matrix}e^{\chi_-}a_1^*&-e^{-\chi_-}b_1\\e^{\chi_-}b_1^*&e^{-\chi_-}a_1\end{matrix}\right))
\left(\begin{matrix}(e^{\chi_+}d_2^*f_1-e^{-\chi_+}c_2^*f_2)_-\\0\end{matrix}\right)$$
where the vanishing of the second entry uses the fact that $c_2(0)=0$. This now equals
$$\left(\begin{matrix}[e^{\chi_-}a_1^*(e^{\chi_+}d_2^*f_1-e^{-\chi_+}c_2^*f_2)_-]_{0+}\\
[e^{\chi_-}b_1^*(e^{\chi_+}d_2^*f_1-e^{-\chi_+}c_2^*f_2)_-]_{+}\end{matrix}\right)=0 $$
This proves the second statement.

The third statement is equivalent to $C(k_1^* e^{\chi_-h_1})B(e^{\chi_{0+}h_1}k_2)=0$. This is a similar calculation.
\end{proof}

\begin{lemma}\label{injectivelemma} $A(k_1^* e^{\chi_-h_1})$ and $A(e^{\chi_{0+}h_1}k_2)$ are injective on their domains, and similarly for $A_1$. \end{lemma}

\begin{proof} The four statements are all proved in the same way. We consider the second assertion concerning $A$. Suppose that
$$
A(\left(\begin{matrix}e^{\chi_+}d_2^*&-e^{-\chi_+}c_2^*\\e^{-\chi_+}c_2&e^{-\chi_+}d_2\end{matrix}\right))
\left(\begin{matrix}f_1\\f_2\end{matrix}\right)=0$$
This implies
$$\left(\begin{matrix}[e^{\chi_+}(d_2^*f_1-c_2^*f_2)]_+\\e^{-\chi_+}(c_2f_1+d_2f_2)\end{matrix}\right))=0 $$
The second component implies $c_2f_1+d_2f_2=0$, and this implies
$$\left(\begin{matrix}f_1\\f_2\end{matrix}\right)=g\left(\begin{matrix}d_2\\-c_2\end{matrix}\right)$$
where $g$ is holomorphic in the disk. Plug this into the first component to obtain
$$[e^{\chi_+}g(d_2d_2^*+c_2c_2^*)]_+=[e^{\chi_+}g]_+=0$$
which implies $g=0$. Thus $f=0$.

\end{proof}

Now assume that (c) of Conjecture \ref{introtheorem6} holds. The lemmas imply that the Toeplitz operator $A(g)$
and the shifted Toeplitz operator $A_1(g)$ are injective. Since these operators are Fredholm, they are invertible. Hence (c) implies (a) and (b).

Now assume (a) and (b). We define $k_1,k_2$ and $\chi$ using the explicit formulas in the proof of Theorem \ref{L2theorem}. Note it is essential that we use these explicit formulas, because (as we saw in the last section) the existence of singular inner functions implies that root subgroup factorization is not unique in general. In particular $\chi=\chi_-+\chi_0+\chi_-$ has a Fourier series. The formula (\ref{chieqn}) immediately implies that $exp(-\chi_+)\in L^2$. The crux of the matter is to show that if $g\in VMO$ (or more generally $B_p^{1/p}$),
then the factors have the same smoothness property.

Suppose first that $\chi=0$. In this case
$$A(g)A(g)^*=A(k_1^*)A(k_2)A(k_2)^*A(k_1) $$
$$=1-B(k_1)B(k_1)^*-A(k_1^*)B(k_2)B(k_2)^*A(k_1) $$
This implies the following sum is a positive compact operator:
$$ B(k_1)B(k_1)^*+A(k_1^*)B(k_2)B(k_2)^*A(k_1) $$
Does this imply that the two summands have to be compact?

\begin{proposition}\label{compactness} Assume $A$ and $B$ are positive operators on a Hilbert space $H$.

(a) If $A+B$ is finite rank, then $A$ and $B$
are finite rank.

(b) If $A+B$ is compact (or Schatten p-class),
the $A$ and $B$ are compact (Schatten p-class, respectively).
\end{proposition}

\begin{proof} (a) For $x\in ker(A+B)$,
$$\langle Ax,x\rangle +\langle Bx,x\rangle=0$$
together with polarization, this implies that $\langle Ax,y\rangle=0$ for $x,y\in ker(A+B)$.
$ker(A+B)^{\perp}$ is finite dimensional. So the range of $A$ is contained in the finite dimensional
subspace
$$ker(A+B)^{\perp}+A(ker(A+B)^{\perp})$$
and similarly for $B$. This proves $A$ and $B$ are finite rank.

(b) Given $n$, let $K_n$ ($P_n$) denote the closed subspace (and the corresponding orthogonal projection) spanned by eigenvectors corresponding to eigenvalues
$ \lambda$ for $A+B$ with $\lambda<1/n$. $K_n$ is $A+B$ invariant and $|A+B|_{K_n}<1/n$. The orthogonal complement of $K_n$ is finite dimensional. Because $\langle Ax,x\rangle\le\langle(A+ B)x,x\rangle$ for $x\in K_n$ and $A$ is positive, the norm for $|P_nAP_n|<1/n$. Define $A_n=A-P_nAP_n$. This is a finite rank operator (its range is contained
in $K_n^{\perp}+A K_n^{\perp}$) and $|A_n-A|=|P_nAP_n|<1/n$. This shows that $A$ is a norm limit of finite rank operators. Hence $A$ is compact.

The Schatten p-class claim is done in the same way, using the Schatten p-norm.

\end{proof}

Thus if $\chi=0$, then $g\in VMO$ implies that $k_1,k_2\in VMO$.

Now consider the general case,
$$g=k_1^*e^{\chi h_1}k_2=\left(\begin{matrix}a_1^*e^{\chi}d_2^*-b_1e^{-\chi}c_2&  -a_1^*e^{\chi}c_2^*-b_1e^{-\chi}d_2\\ b_1^*e^{\chi}d_2^*+a_1e^{-\chi}c_2  & -b_1^*e^{\chi}c_2^*+a_1e^{-\chi}d_2 \end{matrix}\right) $$

The following is a basic gap in this section, and we will simply assume its truth.

\begin{lemma}\label{deformation} There exists a deformation
$\chi_t:S^1\to i\mathbb R$ with $\chi|_{t=0}=\chi$, $\chi_t\in VMO(S^1,i\mathbb R)$ and
and $g_{t}:=k_1^*e^{\chi_t h_1}k_2 \in VMO(S^1,SU(2))$, for some $t>0$.
\end{lemma}

\begin{remark} By definition $\chi$ has a linear triangular factorization, $\chi=\chi_-+\chi_0+\chi_+$ (with $\chi_-=-\chi_+$ and $\chi_0\in i\mathbb R$), and by assumption $exp(-\chi_+)\in L^2$, i.e. $exp(-H(\chi))$ is integrable, where
$H(\chi)=\chi_++\chi_+^*):S^1 \to \mathbb R$ is the Hilbert transform of $\chi$.  The upshot of this is
that we can easily deform $\chi$ so that it is smooth (or simply $VMO$). The difficult point is to guarantee that for some $t$, $g_t$ will be $VMO$.

Consider the operator identities
$$B(k_1^*e^{\chi}k_2)=A(k_1^*)A(e^{\chi})B(k_2)+A(k_1^*)B(e^{\chi})D(k_2)
+B(k_1^*)C(e^{\chi})B(k_2)+B(k_1^*)D(e^{\chi})C(k_2)$$
and
$$C(k_1^*e^{\chi}k_2)=C(k_1^*)A(e^{\chi})A(k_2)+C(k_1^*)B(e^{\chi})C(k_2)
+D(k_1^*)C(e^{\chi})A(k_2)+D(k_1^*)D(e^{\chi})C(k_2)$$
When we add in the deformation, the second and third factors will be compact. The crux of the matter is to choose the deformation
so that the sum of the first and fourth factors is compact.

\end{remark}

Lemma \ref{Toeplitzfactorization} and some algebraic manipulations imply the following
lemma (without the use of Lemma \ref{Toeplitzfactorization}, the expression in the following lemma would have eight terms).

\begin{lemma}$A(g_t)$ equals the sum of four terms
$$A(k_1^*)A(e^{\chi_t h_1})A(k_2)+B(k_1^*)C(e^{\chi_t h_1})A(k_2)$$
$$+A(k_1^*)B(e^{\chi_t h_1})C(k_2)
+B(k_1^*)C(e^{\chi_t h_1})A(e^{-\chi_t h_1})B(e^{\chi_t h_1})C(k_2)$$
The last three terms are compact for $t>0$.
\end{lemma}

This implies that $A(g_t)A(g_t^{-1})$ will be the sum of 16 terms

$$=A(k_1^*)A(e^{\chi_t h_1})A(k_2)A(k_2^*)A(e^{-\chi_t h_1})A(k_1)+
A(k_1^*)A(e^{\chi_t h_1})A(k_2)A(k_2)^*B(e^{-\chi_t h_1})C(k_1)$$
$$+A(k_1^*)A(e^{\chi_t h_1})A(k_2)B(k_2^*)C(e^{-\chi_t h_1})A(k_1)
+A(k_1^*)A(e^{\chi_t h_1})A(k_2)B(k_2^*)C(e^{-\chi_t h_1})A(e^{\chi_t h_1})B(e^{-\chi_t h_1})C(k_1)$$
$$+B(k_1^*)C(e^{\chi_t h_1})A(k_2)A(k_2^*)A(e^{-\chi_t h_1})A(k_1)+
B(k_1^*)C(e^{\chi_t h_1})A(k_2)A(k_2)^*B(e^{-\chi_t h_1})C(k_1)$$
$$+B(k_1^*)C(e^{\chi_t h_1})A(k_2)B(k_2^*)C(e^{-\chi_t h_1})A(k_1)
+B(k_1^*)C(e^{\chi_t h_1})A(k_2)B(k_2^*)C(e^{-\chi_t h_1})A(e^{\chi_t h_1})B(e^{-\chi_t h_1})C(k_1)$$
$$+A(k_1^*)B(e^{\chi_t h_1})C(k_2)A(k_2^*)A(e^{-\chi_t h_1})A(k_1)+A(k_1^*)B(e^{\chi_t h_1})C(k_2)A(k_2)^*B(e^{-\chi_t h_1})C(k_1)$$
$$+A(k_1^*)B(e^{\chi_t h_1})C(k_2)B(k_2^*)C(e^{-\chi_t h_1})A(k_1)
+A(k_1^*)B(e^{\chi_t h_1})C(k_2)B(k_2^*)C(e^{-\chi_t h_1})A(e^{\chi_t h_1})B(e^{-\chi_t h_1})C(k_1)$$
$$+B(k_1^*)C(e^{\chi_t h_1})B(e^{\chi_t h_1})C(k_2)A(k_2^*)A(e^{-\chi_t h_1})A(k_1)+
B(k_1^*)C(e^{\chi_t h_1})B(e^{\chi_t h_1})C(k_2)A(k_2)^*B(e^{-\chi_t h_1})C(k_1)$$
$$+B(k_1^*)C(e^{\chi_t h_1})B(e^{\chi_t h_1})C(k_2)B(k_2^*)C(e^{-\chi_t h_1})A(k_1)$$
$$+B(k_1^*)C(e^{\chi_t h_1})B(e^{\chi_t h_1})C(k_2)B(k_2^*)C(e^{-\chi_t h_1})A(e^{\chi_t h_1})B(e^{-\chi_t h_1})C(k_1)  $$

For $t>0$ all of the terms, with the exception of the first, are compact, because $e^{\chi_t}$ is $VMO$.
The first term is

$$   A(k_1^*)A(e^{\chi_t h_1})A(k_2)A(k_2^*)A(e^{-\chi_t h_1})A(k_1)
=A(k_1^*)A(e^{\chi_t h_1})(1-B(k_2)B(k_2)^*)A(e^{-\chi_t h_1})A(k_1)=$$
$$A(k_1^*)(1-B(e^{\chi_t h_1}B(e^{\chi_t h_1})^*)A(k_1)-A(k_1^*)A(e^{\chi_t h_1})B(k_2)B(k_2)^*A(e^{-\chi_t h_1})A(k_1)=$$
$$1-B(k_1^*)B(k_1^*)^*-A(k_1^*)B(e^{\chi_t h_1}B(e^{\chi_t h_1})^*)A(k_1)-A(k_1^*)A(e^{\chi_t h_1})B(k_2)B(k_2)^*A(e^{-\chi_t h_1})A(k_1)$$
Since $B(g_t)B(g_t)^*$ is positive, this implies that
$$B(k_1^*)B(k_1^*)^*+A(k_1^*)B(e^{\chi_t h_1}B(e^{\chi_t h_1})^*)A(k_1)+A(k_1^*)A(e^{\chi_t h_1})B(k_2)B(k_2)^*A(e^{-\chi_t h_1})A(k_1)$$
is positive. Proposition \ref{compactness} now implies that $B(k_1)$ and $B(k_2)$ are compact, hence $k_1$ and $k_2$
are VMO. This now implies that $e^{\chi}$ is VMO. Lemma \ref{abelian} implies that $\chi$ is VMO. This completes the proof of the theorem.
\end{proof}

\subsection{Coordinates for $VMO(S^1,SU(2))$}\label{VMOII}

Theorem \ref{introtheorem6} implies the following

\begin{corollary} $VMO(S^1,SU(2))$ is a topological manifold, where $(y,\chi,x)$ is a topological coordinate system for the open set of loops in $VMO(S^1,SU(2))$ with invertible $A$ and $A_1$.
\end{corollary}

\subsection{Lower Strata}

As in the $W^{1/2}$ case one can use finite codimensional conditioning to parameterize the lower strata.
There are two cosmetic changes. First one cannot use the root subgroup coordinates $\eta,\chi,\zeta$, except in the $W^{1/2}$ case. Secondly instead of the Hilbert-Schmidt Grassmannian (and flag space), one uses the Grassmannians corresponding to the p-Schatten class ideals. Otherwise the statements of Theorems \ref{W12lowerstrata} and \ref{W12homotopy} remain valid for VMO loops. In particular the inclusions
$$L_{fin}SU(2) \to C^{\infty}(S^1,SU(2)) \to VMO(S^1,SU(2))$$ are homotopy equivalences.

\section{Appendix: Combinatorial Formulas}

In this appendix we will recall how, given $k_2$, one solves for the $zeta$ variables, and how one expresses
$x_1$ in terms of $\zeta$. The formulas we will
discuss are an incremental improvement of, but not a final word on, those appearing in \cite{BP}, where we considered root subgroup factorization for loops into $SL(2,\mathbb C)$. In this appendix we are considering the complex case.

We briefly recall the basic background. Let $L_{fin}SL(2,\mathbb C)$ denote the group
consisting of functions $S^1 \to SL(2,\mathbb C)$ having finite Fourier series, with pointwise
multiplication. For example suppose that $\zeta=(\zeta^-,\zeta^+) \in \mathbb C^2$, $\zeta^-\zeta^+\ne 1$, $n\in
\mathbb N$, and choose a square root
$$\mathbf a(\zeta):= (1+\zeta^-\zeta^+)^{-1/2}$$
Then the function
\begin{equation}\label{basicexample} S^1 \to SL(2,\mathbb C):z \to
\mathbf a(\zeta) \left(\begin{matrix} 1&
-\zeta^- z^{-n}\\
\zeta^+ z^n&1\end{matrix} \right)\end{equation}
is in $L_{fin}SL(2,\mathbb C)$.

The following is from \cite{BP}, with one slight change: we have inserted
a minus sign in front of $\zeta^-$, as in (\ref{basicexample}), in order to eliminate a profusion of signs
in some combinatorial formulas.

\begin{theorem}\label{SU(2)theorem1} Suppose that $g_1 \in L_{fin}SL(2,\mathbb C)$. Consider the following three conditions:

(I.1) $g_1$ is of the form
$$g_1(z)=\left(\begin{matrix} a_1(z)&b_1(z)\\
c_1^*(z)&d_1^*(z)\end{matrix} \right),\quad z\in S^1,$$ where $a_1,b_1,c_1$ and
$d_1$ are polynomials in $z$, and $a_1(0)=d_1^*(\infty)> 0$.

(I.2) $g_1$ has a factorization of the form
$$g_1(z)=\mathbf a(\eta_n)\left(\begin{matrix} 1& \eta^+_nz^n\\
\eta^-_nz^{-n}&1\end{matrix} \right)..\mathbf a(\eta_0)\left(\begin{matrix} 1&
\eta^+_0\\
\eta^-_0&1\end{matrix} \right),$$ for some finite subset
$\{\eta_0,..,\eta_n\} \subset \{\eta\in\mathbb C^2:1-\eta^-\eta^+\ne 0\}$.

(I.3) $g_1$ and $g_1^{-*}$ have triangular factorizations of the form
$$\left(\begin{matrix} 1&0\\
\sum_{j=0}^n \bar y_jz^{-j}&1\end{matrix} \right)\left(\begin{matrix} \mathbf a_1&0\\
0&\mathbf a_1^{-1}\end{matrix} \right)\left(\begin{matrix} \alpha_1 (z)&\beta_1 (z)\\
\gamma_1 (z)&\delta_1 (z)\end{matrix} \right),$$ where
the third factor is a polynomial in $z$ which is unipotent upper
triangular at $z=0$, and $\mathbf a_1$ is a positive constant.

(I.1) and (I.3) are equivalent. (I.2) implies (I.1) and (I.3). The converse holds generically, in the sense
that the $\eta$ variables are rational functions of the Laurent coefficients
for $b_1/a_1$ and $c_1^*/d_1^*$.

Similarly, for $g_2 \in L_{fin}SL(2,\mathbb C)$, consider the following three conditions:

(II.1) $g_2$ is of the form
$$g_2(z)=\left(\begin{matrix} a^*_2(z)&b^*_2(z)\\
c_2(z)&d_2(z)\end{matrix} \right),\quad z\in S^1,$$ where $a_2,b_2,c_2$ and $d_2$ are polynomials
in $z$, $c_2(0)=b_2(0)=0$, and $a_2^*(\infty)=d_2(0)> 0$.

(II.2) $g_2$ has a factorization of the form
$$g_2(z)=\mathbf a(\zeta_n)\left(\begin{matrix} 1&-\zeta_n^-z^{-n}\\
\zeta_n^+z^n&1\end{matrix} \right)..\mathbf a(\zeta_1)\left(\begin{matrix} 1&
-\zeta_1^-z^{-1}\\
\zeta_1^+z&1\end{matrix} \right),$$ for some finite subset
$\{\zeta_1,..,\zeta_n\} \subset \{\zeta\in\mathbb C^2:1-\zeta^-\zeta^+\ne 0\}$.

(II.3) $g_2$ and $g_2^{-*}$ have triangular factorizations of the form
$$\left(\begin{matrix} 1&\sum_{j=1}^n \bar x_jz^{-j}\\
0&1\end{matrix} \right)\left(\begin{matrix}\mathbf a_2&0\\0&\mathbf a_2^{-1}\end{matrix}\right)\left(\begin{matrix} \alpha_2 (z)&\beta_2 (z)\\
\gamma_2 (z)&\delta_2 (z)\end{matrix} \right), $$ where the third factor is a polynomial in $z$ which is unipotent
upper triangular at $z=0$, and $\mathbf a_2$ is a positive constant.

(II.1) and (II.3) are equivalent. (II.2) implies (II.1) and (II.3). The converse holds generically, in the sense
that the $\zeta$ variables are rational functions of the Laurent coefficients for $c_2/d_2$ and $b_2^*/a_2^*$.

\end{theorem}

We will focus on the second set of equivalences. Given $g_2$ to find the "root subgroup coordinates"
$\zeta^{\pm}$, one needs the Taylor series expansions for $c_2(g_2)/d_2(g_2)$ and for $c_2(g_2^{-*})/d_2(g_2^{-*})=
-b_2(g_2)/a_2(g_2)$ (In the special cases when $g_2$ has values in $SU(2)$ or $SU(1,1)$, these two functions are essentially the same). We also want to find the Taylor expansion for $x(z)$, expressed in terms of $\zeta$.

\begin{proposition} In reference to $g_2$ in II.1-3,

(a)
$$(c_2(g_2)/d_2(g_2))(z)=(\gamma_2/\delta_2)(z)=\sum_{n=1}^{\infty}\xi_n z^n, \text{ where } \xi_n=\sum_{k=1}^{n} \zeta_{k}^+\prod_{j=0}^{k-1}(1+\zeta_j^-\zeta_j^+)p_{n,k}(\zeta_1^{\pm},..,\zeta_{k}^{\pm}) $$
$p_{1,1}=\zeta^+_1$ and for $n>1$
$$p_{n,k}=\sum c_{ij}\left(\zeta^-_{j_1}(\zeta^+_{i_1})\right)...\left(\zeta^-_{j_r}(\zeta^+_{i_r})\right) $$
where $1\le j_s<i_s\le k$ for $s=1,..,r$, $\sum_{s=1}^{r}(i_s- j_s)=n-k$, and $c_{ij}$ is a positive integer; in particular
$$\xi_n=(\zeta^+_n)\prod_{s=1}^{n-1}(1+\zeta_s^-\zeta_s^+)+polynomial(\zeta^-_s,\zeta_s^+,s<n)$$

(b)
$$x^*(z)=(\gamma_2^*/\delta_2)(z)=\sum_{n=1}^{\infty}x_1^*(\zeta_n,\zeta_{n+1},...)z^{-n} $$
where
$$x_1^*(\zeta_1^{\pm},\zeta_2^{\pm},...)=\sum_{m=1}^{\infty} \zeta_m\prod_{j=m+1}^{\infty}(1+\zeta_j^-\zeta_j^+)s_{m}(\zeta_m^{\pm},\zeta_{m+1}^{\pm},...) $$
and
$$s_{m}=
\sum C_{ij}\left(\zeta^-_{j_1}(\zeta^+_{i_1})\right)...\left(\zeta^-_{j_r}(\zeta^+_{i_r})\right) $$
where $1\le j_s<i_s\le k$ for $s=1,..,r$, $\sum_{s=1}^{r}(i_s- j_s)=n-k$, and $C_{ij}$ is a positive integer.

\end{proposition}

In (a) and (b) it is not known how to compute the positive integers $c_{ij}$ and $C_{ij}$, respectively.
It is easy to generate intriguing formulas, but pinning this down remains a vexing problem.

\section{Appendix: Is $\theta_+$ a Coordinate?}

There may be a more direct way to parameterize the top stratum of the quotient
$W^{1/2}(S^1,SU(2))/SU(2)$ which does not involve root subgroup coordinates.

Given $g\in L^{\infty}(S^1,SL(2,\mathbb C))$, the existence of a Riemann-Hilbert
factorization $g=g_-g_0g_+$ is equivalent to a matrix factorization of the
multiplication operator $M_g$,
\begin{equation}\left(\begin{matrix} A&B\\
C&D\end{matrix} \right)=\left(\begin{matrix} 1&0\\
Z&1\end{matrix} \right)\left(\begin{matrix} A&0\\
0&D-ZB\end{matrix} \right)\left(\begin{matrix} 1&W\\
0&1\end{matrix} \right).\label{ 4.6}\end{equation}
where $Z=CA^{-1}$ and $W=A^{-1}B$. The following is straightforward.

\begin{lemma}For $g$ having a Riemann-Hilbert factorization,
$Z(g)=Z(g_{-}g_0)=Z(g_-)$, $W(g)=W(g_0g_{+})=W(g_+)$ and
$$A(g^{-1})A(g)=(1+W(g)Z(g^{-1}))^{-1}$$
Thus for $g:S^1 \to SU(2)$,
$$det(A(g)A(g^{-1}))=det(1+W(g)W(g)^{*})^{-1}$$
and $g\in W^{1/2}$ iff $W(g)=W(g_+)$ is Hilbert-Schmidt.
\end{lemma}

It follows that $W(g)$ only depends on $\theta_+=g_+^{-1}\partial g_+\in H^1(\Delta)$.

\begin{question} Is there a bijective correspondence between the set of cosets $[g]\in W^{1/2}(S^1,SU(2))/SU(2)$
such that $g$ has a Riemann-Hilbert factorization and the Hilbert space of $\theta_+\in H^1(\Delta)$ which are square integrable, where $g\to \theta_+=g_+^{-1}\partial g_+$?
\end{question}

\begin{remarks} (a) It is doubtful that one can express this criterion in terms of $g_+$: (1) $W^{1/2}$ and $L^{\infty}\cap W^{1/2}$ are not decomposing algebras, hence one cannot expect $g\in W^{1/2}$ to correspond to $g_+\in W^{1/2}$; and (2) for differential equations of the type $g_+^{-1}\partial g_+=\theta_+$,
the case we are considering is critical, and $\int_0^z \theta_+\in W^{1/2}(S^1)$ does not imply $g_+\in W^{1/2}(S^1)$ and vice versa.

(b) It is natural to search for a criterion in terms of $\theta_+$, because $PSU(1,1)$ acts equivariantly on
the mapping from the set of $g$ which have Riemann-Hilbert factorizations to square integrable $\theta_+\in H^1(\Delta)$ (the latter action is isometric).

(c) For $k_2(\zeta)\in W^{1/2}(S^1,SU(2))$ and
$$ k_2=\left(\begin{matrix} 1&x^*\\
0&1\end{matrix} \right)\left(\begin{matrix}\mathbf a_2&0\\
0&\mathbf a_2^{-1}\end{matrix} \right)\left(\begin{matrix} \alpha_2 (z)&\beta_2 (z)\\
\gamma_2 (z)&\delta_2 (z)\end{matrix} \right)$$
$\partial x,\partial \gamma_2,\partial \delta_2\in H^1(\Delta)$ and square integrable, but it is not clear that
the same is true for
$$
\left(\begin{matrix} \delta_2 &-\beta_2 \\
-\gamma_2&\alpha_2\end{matrix} \right)\left(\begin{matrix} \partial\alpha_2&\partial\beta_2\\
\partial\gamma_2&\partial\delta_2\end{matrix} \right)=\left(\begin{matrix} \delta_2\partial\alpha_2+\beta_2\partial \gamma_2&\delta_2\partial\beta_2-\beta_2\partial \delta_2\\ -\gamma_2\partial\alpha_2+\alpha_2\partial \gamma_2&-\gamma_2\partial\beta_2+\alpha_2\partial \delta_2\end{matrix} \right)\in H^1(\Delta) $$
A relevant fact is that $\gamma_2,\delta_2$ are bounded, but this does not seem to help.
\end{remarks}

In the rest of this subsection we want to write out the dependence of $W$ on $\theta_+$.
We can picture $W$ as a block matrix, $W=(W_{i,-j})$,
where each block $W_{i,-j}\in \mathcal L(\mathbb C^2)$,
and the index $i\ge 0$ denotes the row
(starting from the bottom) and $j\ge 1$ denotes the column
(starting from the left). If
\begin{equation}g_{+}=1+g_1z+g_2z^2+..\label{ 4.16}\end{equation}
where $g_j\in \mathcal L(\mathbb C^2)$, then
\begin{equation}g_{+}^{-1}=1+(-g_1)z+(-g_2+g_1^2)z^2+(-g_3+g_1g_2+g_2g_1-g_1^3)
z^3+..\end{equation}
\begin{equation}\label{ 4.17}\end{equation}
and in general
\begin{equation}(g_{+}^{-1})_n=\sum (-1)^lg_{i_1}..g_{i_l}\label{ 4.18}\end{equation}
where the sum is over all $positive$ multi-indices
$I=(i_1,..,i_l)$ of order $n$, i.e. $i_m>0$ and
$i_1+..+i_l=n$. By $(4.15)$,
\begin{equation}W_{i,-j}=1\cdot g_{i+j}+(g_{+}^{-1})_1g_{i+j-1}+..+(g_{+}^{-1})_
ig_j\end{equation}
\begin{equation}=\sum_{n=j}^{i+j}\sum (-1)^lg_{i_1}..g_{i_l}g_n\label{ 4.19}\end{equation}
where given $n$, the sum is over all positive multi-indices
of order $i+j-n$. This can also be written as
\begin{equation}W_{i,-j}=\sum (-1)^{l+1}g_{i_1}..g_{i_l}\label{ 4.20}\end{equation}
where the sum is now over all positive multi-indices of
order $i+j$ satisfying $i_l\ge j$. Thus, in terms of the
representation dependent expression $(\ref{4.16})$ for $g_{+}$,
$W$ has the form
\begin{equation}\left(\begin{matrix} ...&...\\
g_3-g_1g_2-g_2g_1+g_1^3&g_4-g_1g_3-g_2g_2+g_1^2g_2&...&\\
g_2-g_1g_1&g_3-g_1g_2&g_4-g_1g_3\\
g_1&g_2&g_3&..\end{matrix} \right)\label{4.21}\end{equation}

Now write
\begin{equation}\theta_{+}=(\theta_1+\theta_2z+..)dz\in H^1(D,\mathfrak g),\label{ 4.22}\end{equation}
where $\theta_i\in \mathfrak g$. Since $g_{+}$ is the solution of the integral
equation
\begin{equation}g_{+}(z)=1+\int_0^zg_{+}(w)\theta (w),\quad g_{+}(0)=1,\label{ 4.23}\end{equation}
it can be expressed in terms of iterated integrals:
\begin{equation}g_{+}(z)=1+\int\theta +\int \{(\int\theta )\theta \}+\int \{(\int
\theta )\int\end{equation}
\begin{equation}=1+g^{(1)}(\theta )+g^{(2)}(\theta )+..\label{4.24}\end{equation}
where
\begin{equation}g^{(n)}(\theta )=\int g^{(n-1)}(\theta )\theta =\sum\frac 1{i_1}
..\frac 1{i_1+..i_n}\theta_{i_1}..\theta_{i_n}z^{\vert I\vert}\label{ 4.25}\end{equation}
and the sum is over all positive multiindices
$I=(i_1,..,i_n)$.

Given a positive multi-index $I=(i_1,i_2,..,i_l)$, define
\begin{equation}c(I)=\frac 1{i_1}\frac 1{i_1+i_2}..\frac 1{i_1+..+i_l}.\label{ 4.26}\end{equation}
Observe that there is a bijective correspondence between
positive multi-indices $I$ of length $n$ and subsets of
$S\subset \{1,..,n-1\}$:  A multi-index $I$ induces a strictly
increasing sequence
\begin{equation}\lambda_1=i_1<\lambda_2=i_1+i_2<..<\lambda_l=i_1+..+i_l=n\label{ 4.27}\end{equation}
which is uniquely determined by the complement
\begin{equation}S=\{1,..,n\}\setminus \{\lambda_1,..,\lambda_l\}\subset \{1,..,
n-1\}.\label{ 4.28}\end{equation}
In terms of $S$, the $i_j$ are of the form
\begin{equation}i_j=1+\vert S_j\vert ,\label{ 4.29}\end{equation}
where $S_j$ is the $j$th connected component of $S$, and two
integers are connected if they are adjacent.  We can
then write
\begin{equation}c(I)=\frac {\prod_S\lambda}{n!}\label{ 4.30}\end{equation}
We have
\begin{equation}g_n=\sum c(I)\theta_I=\sum c(I)\theta_{i_1}..\theta_{i_l}\label{ 4.31}\end{equation}
where the sum is over all positive multi-indices of
order $n$. Plugging this into $(4.20)$ implies the following
formula.

\begin{proposition}As a function of $\theta_{+}\in H^1
(D,\mathfrak g)$,
for $i\ge 0$ and $j\ge 1$,
\begin{equation}W_{i,-j}=\sum C(I)\theta_{i_1}..\theta_{i_l},\end{equation}
where $I$ ranges over all positive multi-indices of order
$i+j$, and
\begin{equation}C(I)=\sum (-1)^{l+1}c(I_1)..c(I_l)\end{equation}
\begin{equation}=\sum (-1)^{l+1}\frac {\prod_{S_1}\lambda_1}{n!}..\frac {\prod_{
S_l}\lambda_l}{n!}\end{equation}
where the sum is over all ways of representing $I$ as a
tuple $(I_1,..,I_l)$ with $\vert I_l\vert\ge j$.
\end{proposition}

Thus $W(\theta_{+})$ has the form
\begin{equation}\left(\begin{matrix} ..&..\\
\frac 1{3!}(2\theta_3-\theta_1\theta_2-2\theta_2\theta_1+\theta_1^
3)&..&..&\\
\frac 12(\theta_2-\theta_1\theta_1)&\frac 1{3!}(2\theta_3-[\theta_
1,\theta_2]-2\theta_1^3)&..\\
\theta_1&\frac 12(\theta_2+\theta_1\theta_1)&\frac 1{3!}(2\theta_
3+2\theta_1\theta_2+\theta_2\theta_1+\theta_1^3)&..\end{matrix} \right
)\end{equation}

\end{document}